\documentclass[reqno]{amsart}

\usepackage{amsmath,amssymb,amsthm,amscd}
\usepackage{graphics}
\usepackage[mathscr]{eucal}
\usepackage{color}
\usepackage{epic,eepic}
\usepackage[all]{xy}
\usepackage{comment}

\newtheorem{theorem}{Theorem}
\newtheorem{lemma}[theorem]{Lemma}

\newtheorem{proposition}[theorem]{Proposition}

\theoremstyle{definition}

\newtheorem{remark}[theorem]{Remark}

\newcommand{\C}{{\mathbb C}}

\newcommand{\Z}{{\mathbb Z}}

\newcommand{\red}[1]{{\color{red}#1}}

\newcommand{\geh}{\mathfrak{g}}
\newcommand{\heh}{\mathfrak{h}}
\newcommand{\neh}{\mathfrak{n}}
\newcommand{\ep}{\epsilon}

\newcommand{\La}{\Lambda}
\newcommand{\ot}{\otimes}
\newcommand{\pair}[1]{\langle{#1}\rangle}

\begin{document}

\title[]{Parafermionic bases of standard modules for twisted affine Lie algebras of type $A_{2l-1}^{(2)}$, $D_{l+1}^{(2)}$, $E_6^{(2)}$ and $D_4^{(3)}$}
%[\today]

\author{MASATO OKADO}
\address{Department of Mathematics, Osaka City University, Osaka 558-8585, Japan}
\email{okado@osaka-cu.ac.jp}

\author{RYO TAKENAKA}
\address{Department of Mathematics, Osaka City University, Osaka 558-8585, Japan}
\email{m20sa023@zp.osaka-cu.ac.jp}

\begin{abstract}
Using the bases of principal subspaces for twisted affine Lie algebras except $A_{2l}^{(2)}$
by Butorac and Sadowski, we construct bases of the highest weight modules of highest
weight $k\La_0$ and parafermionic spases for the same affine Lie algebras. As a result,
we obtain their character formulas conjectured in \cite{HKOTT}.
\end{abstract}

\maketitle

\section{Introduction}

In 1980's, the attempts to obtain combinatorial bases of the integrable 
highest weight modules or
their (coset) subspaces using vertex operators were initiated in the seminal work of
Lepowsky and Primc \cite{LP}. Let $\geh$ be the simple Lie algebra $sl_2$
and $\heh$ its Cartan subalgebra. Set $\hat{\geh}=\geh\ot\C[t,t^{-1}]\oplus\C c,
\hat{\heh}=\heh\ot\C[t,t^{-1}]\oplus\C c$, where $c$ is the canonical central 
element. Lepowsky and Primc constructed bases of the coset space 
$\hat{\geh}\supset\hat{\heh}$ of integrable highest weight modules build upon vertex 
operator components. This work was generalized by Georgiev to the higher rank case
$\geh=sl_{n+1}$. In \cite{G1}, he first constructed the bases for the pricipal subspaces
studied by Feigin and Stoyanovsky \cite{FS1,FS2}, and then did for the 
$\hat{\geh}\supset\hat{\heh}$ coset subspaces or equivalently the parafermionic 
spaces in \cite{G2}. As a result, he obtained a fermionic character formula of the
integrable highest weight module $L(k_0\La_0+k_j\La_j)$. 

Recently, based on the works on constructing combinatorial bases of the principal 
subspaces for other affine algebras (see e.g.~\cite{B,BK}), Butorac, Ko\v{z}i\'c and 
Primc obtained fermionic formulas for integrable highest weight modules
similar to the above highest weight for all untwisted affine Lie algebras in \cite{BKP},
and thereby proved the so called Kuniba-Nakanishi-Suzuki conjecture \cite{KNS}.
Although the untwisted cases are settled, there are also twisted affine Lie algebras
and a similar conjecture was stated in \cite{HKOTT}. Fortunately,
necessary combinatorial bases for the principal subspaces have already been
obtained recently in \cite{PS1,PS2,BS}. 
The purpose of this paper is to make use of their results to settle the conjecture
in \cite{HKOTT} for the twisted cases. 

This paper is organized as follows. In section 2, we review necessary results. In 
particular, we recall the twisted vertex operator, construction of the level 1 highest
weight representation by it, quasi-particles on higher level representations and the
result by Butorac and Sadowski on the construction of a basis of the principal 
subspace by quasi-particles. We then give our main theorem in section 3 on the
construction of a basis of the standard module $L(k\La_0)$. In section 4, we consider
the so called $\mathcal{Z}$-operator in the twisted case and construct a basis of the
parafermionic space. In the final section 5, we obtain the fermionic character formula 
of the parafermionic space for the twisted cases. Note that the $A_{2l}^{(2)}$ case is not 
included, since it is not treated in \cite{BS} either. We wish to report also in this case
in near future.
%
\begin{comment}
\red{Notational change

\begin{tabular}{|c|c|}
\hline
Old&New\\
\hline
$v$&$r$\\
$\hat{\geh}[\hat{\nu}]$&$\hat{\geh}[\nu]$\\
$L^{\hat{\nu}}(k\La_0)$&$L(k\La_0)$\\
$v_L$&$v_0$\\
$W_{L_k}^{T}$&$W(k\La_0)$\\
$r$&$n$\\
$1/\theta_i$&$\rho_i$\\
\hline
\end{tabular}
}
\end{comment}

\section{Preliminaries}

\subsection{Notation on simple Lie algebras of type ADE}

Let $\geh$ be a complex simple Lie algebra of type $A_{\ell}$, $D_{\ell}$ or $E_\ell$ ($\ell=6,7,8$ for $E_\ell$), and let $\alpha_i~(i=1,2,\ldots,\ell)$ be its simple root.
The root lattice of $\geh$ is given by
\[L=\Z\alpha_1\oplus\cdots\oplus\Z\alpha_\ell.\]
Let $\heh$ be the Cartan subalgebra and $\pair{\cdot,\cdot}$ a nondegenerate invariant symmetric bilinear form on $\geh$.
Using this form, we can introduce a symmetric bilinear form $\pair{\cdot,\cdot}$ also
on $\heh^*$ by identifying an element $h$ of $\heh$ with an element of $\alpha_h$ of $\heh^*$ via $\pair{\alpha,\alpha_h}=\alpha(h)$.
We fix the bilinear form $\pair{\cdot,\cdot}$ on $\heh^*$ so that we have $\pair{\alpha,\alpha}=2$ if $\alpha$ is a root.
We take the labeling of the Dynkin diagrams of our Lie algebras as in \cite{BS}.
Let $\nu$ be the Dynkin diagram automorphism of $A_{2l-1}$, $D_{l+1}$, $E_6$ and $D_4$, where the order $r$ of $\nu$ is $2,2,2$ and $3$, respectively. 
The Dynkin diagrams and the autormorphism $\nu$ for each type are given
in Table \ref{tab:Dynkin1}.
\begin{table} 
\begin{eqnarray*}
A_{2l-1}
&\vcenter{\xymatrix@R=3ex{
*{\circ}<3pt> \ar@{-}[r]_<{1} \ar@/_18pt/[dd]_\nu& *{\circ}<3pt> \ar@{-}[r]_<{2} 
& {} \ar@{.}[r]&{}  \ar@{-}[r]_>{\,\,\,\,l-1} & *{\circ}<3pt> \ar@{-}[dr] \\  &&&&&*{\circ}<3pt>\ar@{}_<{l}  \\
*{\circ}<3pt> \ar@{-}[r]_<{2l-1} \ar@/^18pt/[uu]& *{\circ}<3pt> \ar@{-}[r]_<{2l-2} 
& {} \ar@{.}[r]&{}  \ar@{-}[r]_>{\,\,\,\,l+1} & *{\circ}<3pt> \ar@{-}[ur]}} & \\
D_{l+1}
&\vcenter{\xymatrix@R=3ex{
&&&&&*{\circ}<3pt> \ar@{-}[dl]^<{l} \ar@/^18pt/[dd]^\nu \\
*{\circ}<3pt> \ar@{-}[r]_<{\!\!1} & *{\circ}<3pt> \ar@{-}[r]_<{\!\!2}
& {} \ar@{.}[r]&{} \ar@{-}[r]_>{l-1} &
*{\circ}<3pt> & \\
&&&&&*{\circ}<3pt> \ar@{-}[ul]_<{\!\!l+1} \ar@/_18pt/[uu]}} & \\
E_6
&\vcenter{\xymatrix{
&&*{\circ} \ar@{-}[d]^<{\;\,4} \\
*{\circ}<3pt>\ar@{-}[r]_<{\!\!1} \ar@/_18pt/[rrrr]_\nu&*{\circ}<3pt>\ar@{-}[r]_<{\!\!2} &
*{\circ}<3pt>\ar@{-}[r]_<{\!\!3} &*{\circ}<3pt>\ar@{-}[r]_<{\!\!5} &
*{\circ}<3pt>\ar@{}_<{6} \ar@/^18pt/[llll]}}& \\
D_4
&\vcenter{\xymatrix{
&*{\circ} \ar@{-}[d]^<{\raise3ex\hbox{\scriptsize{\;\,4}}} \ar@/_8pt/[dl] \\
*{\circ}<3pt>\ar@{-}[r]_<{\!\!1}\ar@/_18pt/[rr]_\nu &*{\circ}<3pt>\ar@{-}[r]_<{\!\!2} &
*{\circ}<3pt>\ar@{}_<{3} \ar@/_8pt/[ul]}} 
\end{eqnarray*}
\caption{Dynkin diagrams and the automorphism $\nu$ \label{tab:Dynkin1}}
\end{table}

\begin{remark}
The results in this paper remain valid when we take $\geh$ to be one of 
$A_\ell,D_\ell,E_{6,7,8}$ and $\nu=\mathrm{id}$.
\end{remark}

Let $\eta$ be a primitive $r$th root of unity and set $\eta_0=(-1)^r\eta$.
Following \S2 of \cite{CalLM} (see also \cite{BS,PS1,PS2}), we consider two central extensions of $L$
by $\pair{\eta_0}$, denoted by $\hat{L}$ and $\hat{L}_\nu$,
\[
1\longrightarrow\pair{\eta_0}\longrightarrow\hat{L}\;(\text{ or }\hat{L}_\nu)
\mathop{\longrightarrow}^{-} L\longrightarrow1
\]
where $\bar{\phantom{x}}$ stands for the projection to $L$.
$\hat{L}$ and $\hat{L}_\nu$ differ from each other when $r=3$.
Define the functions $C_0,C:L\times L\rightarrow\C^\times$ by
\[
C_0(\alpha,\beta)=(-1)^{\pair{\alpha,\beta}},\qquad C(\alpha,\beta)=\prod_{j=0}^{r-1}(-\eta^j)^{\pair{\nu^j\alpha,\beta}}.
\]
The central extension $\hat{L}$ (resp. $\hat{L}_\nu$) is defined by using the 
commutator map $C_0$ (resp. $C$), namely, for $a,b\in\hat{L}$ (resp. $\hat{L}_\nu$)
$aba^{-1}b^{-1}=C_0(\bar{a},\bar{b})$ (resp. $C(\bar{a},\bar{b}))$. 
Let each commutator map correspond to the 2-cocyle $\ep_{C_0},\ep_C$. 
Namely, for $\ep_{C_0}$, it satisfies 
\[
\ep_{C_0}(\alpha,\beta)\ep_{C_0}(\alpha+\beta,\gamma)=
\ep_{C_0}(\beta,\gamma)\ep_{C_0}(\alpha,\beta+\gamma),\quad
\frac{\ep_{C_0}(\alpha,\beta)}{\ep_{C_0}(\beta,\alpha)}=C_0(\alpha,\beta).
\]
We choose our 2-cocycle $\ep_{C_0}$ to be
\[\ep_{C_0}(\alpha_i,\alpha_j)=
\begin{cases}
1&{\rm if}~i\leq j\\
(-1)^{\pair{\alpha_i,\alpha_j}}&{\rm if}~i>j
\end{cases}.\]
This 2-cocycle satisfies 
\[
\ep_{C_0}(\alpha,\beta)^2=1,\quad \ep_{C_0}(\alpha,\beta)=\ep_{C_0}(\nu\alpha,\nu\beta).
\]
The 2-cocycles $\ep_{C_0}$ and $\ep_C$ are related by
\begin{align}
\label{1}
\ep_{C_0}(\alpha,\beta)=\left(\prod_{-\frac{r}{2}<j<0}(-\eta^{-j})^{\pair{\nu^{j}\alpha,\beta}}\right)\ep_C(\alpha,\beta).
\end{align}
Using the 2-cocycle $\ep_C$, we obtain a normalized section $e:L\rightarrow \hat{L}_\nu$ by
\begin{eqnarray}
\nonumber e:L&\rightarrow&\hat{L}_\nu\\
\nonumber \alpha&\mapsto&e_{\alpha}
\end{eqnarray}
with $e_0=1$, $\overline{e_{\alpha}}=\alpha$ and $e_{\alpha}e_{\beta}=\ep_C(\alpha,\beta)e_{\alpha+\beta}$.
According to \cite{CalLM}, there exists an automorphism $\hat{\nu}$ of $\hat{L}_\nu$ such that
\begin{equation} \label{nu hat}
\overline{\hat{\nu}a}=\nu\bar{a},\qquad\hat{\nu}a=a~{\rm if}~\nu\bar{a}=\bar{a}.
\end{equation}
$\hat{\nu}$ is also an automorphism of $\hat{L}$ satisfying \eqref{nu hat}.
We have $\hat{\nu}^r=1$.

For $j\in\Z$, set
\[\heh_{(j)}=\{h\in\heh\mid\nu h=\eta^jh\}\subset\heh,\]
so that we have
\[\heh=\bigoplus_{j\in\Z/r\Z}\heh_{(j)}.\]
Here we identify $\heh_{(j\text{ mod }r)}$ with $\heh_{(j)}$.
Associated to this decomposition, we define a Lie algebra
\[\hat{\heh}[\nu]=\bigoplus_{m\in\frac1{r}\Z}\heh_{(rm)}\ot t^m\oplus\C c\]
with Lie bracket given by
\[[\alpha\ot t^m,\beta\ot t^n]=\pair{\alpha,\beta}m\delta_{m+n,0}c,\quad [\hat{\heh[}\nu],c]=0\]
for $m,n\in\frac1{r}\Z$ and $\alpha\in\heh_{(rm)}$ and $\beta\in\heh_{(rn)}$.
Consider the subalgebras $\hat{\heh}[\nu]^{\pm}=\bigoplus_{\pm m>0}\heh_{(rm)}\ot t^m$. Then $\hat{\heh}[\nu]_{\frac1{r}\Z}=\hat{\heh}[\nu]^+\oplus\hat{\heh}[\nu]^-\oplus\C c$ is a Heisenberg subalgebra
of $\hat{\heh}[\nu]$. We introduce the induced $\hat{\heh}[\nu]$-module
\[
S[\nu]=U(\hat{\heh}[\nu])\ot_{U(\bigoplus_{m\ge0}\heh_{(rm)}\ot t^m\oplus\C c)}\C,
\]
where $\bigoplus_{m\ge0}\heh_{(rm)}\ot t^m$ acts trivially on $\C$ and $c$ as 1.
$S[\nu]$ is an irreducible $\hat{\heh}[\nu]_{\frac1{r}\Z}$-module, linearly isomorphic to
the symmetric algebra $S(\hat{\heh}[\nu]^-)$.

\subsection{Twisted module and twisted vertex operator} \label{subsec:TVO}

Let $P_j$ be the projection from $\heh$ onto $\heh_{(j)}$ for $j\in\Z/r\Z$.
Following \cite{CalLM} and \cite{L}, we set
\[N=(1-P_0)\heh\cap L=\{\alpha\in L\mid \pair{\alpha,\heh_{(0)}}=0\}.\]
Explicitly, we have 
\[
N=\sum_{i=1}^\ell\Z(\alpha_i-\nu\alpha_i)
\]
when $r=2$, and 
\[
N=\{r_1\alpha_1+r_3\alpha_3+r_4\alpha_4\in L\mid r_1+r_3+r_4=0\}
\]
when $r=3$. See \cite{BS}.

Let $\hat{N}$ be the subgroup of $\hat{L}_\nu$ obtained by pulling back the 
subgroup $N$ of $L$. Then, by Proposition 6.1 in \cite{L}, there is a unique
homomorphism
\[
\tau:\hat{N}\longrightarrow \C^\times
\]
such that 
\[
\tau(\eta_0)=\eta_0,\quad \tau(a\hat{\nu}a^{-1})=\eta^{-\pair{\sum_j\nu^j
\bar{a},\bar{a}}/2}
\]
for $a\in\hat{L}_\nu$. Let $\C_\tau$ be the one-dimensional $\hat{N}$-module
$\C$ with this character $\tau$ and consider the induced $\hat{L}_\nu$-module
\[
U_T=\C[\hat{L}_\nu]\ot_{\C[\hat{N}]}\C_\tau\simeq\C[L/N]\simeq\C[P_0L].
\]
Here the last isomorphism is induced from the projection. The readers are warned 
that $P_0L\not\subset L$ but $P_0L\subset\frac1{r}L$. We have 
\[
U_T=\coprod_{\alpha\in P_0L}(U_T)_\alpha,
\]
where $(U_T)_\alpha=\{u\in U_T\mid hu=\pair{h,\alpha}u\text{ for }h\in\heh_{(0)}\}$ and
\[
a\cdot (U_T)_\alpha\subset (U_T)_{\alpha+\bar{a}_{(0)}}\quad\text{for }a\in \hat{L}_\nu,
\alpha\in P_0L.
\]
$\heh_{(0)}$ acts on $U_T$ as 
\begin{align}
%&a\cdot b\ot t=ab\ot t,\\
&\alpha\cdot b=\pair{\alpha,\bar{b}}b
\end{align}
for $\alpha\in\heh_{(0)},b\in\hat{L}_\nu$ and we set $[\alpha,a]=
\pair{\alpha,\bar{a}}a$. We also define the action of $z^\alpha$ by
\begin{equation}
z^\alpha\cdot b=z^{\pair{\alpha,\bar{b}}}b
\end{equation}
and $\eta^\alpha$ by $\eta^\alpha\cdot b=
\eta^{\pair{\alpha,\bar{b}}}b$. Then for $a\in\hat{L}_\nu$, we have $z^\alpha a=az^{\alpha+\pair{\alpha,\bar{a}}},
\eta^\alpha a=a\eta^{\alpha+\pair{\alpha,\bar{a}}}$. Moreover, as operators on $U_T$,
we have
\begin{equation} \label{nu hat 2}
\hat{\nu}a=a\eta^{-\sum\nu^p\bar{a}-\pair{\sum\nu^p\bar{a},\bar{a}}/2}.
\end{equation}
%
\begin{comment}
Define a $\Z$-grading on $U_T$ by
\begin{equation}
\text{deg}(b\ot t)=-\frac12\pair{\bar{b},\bar{b}}
\end{equation}
calculated by an operator $d$. 
\end{comment}
%
$U_T$ then becomes an $\tilde{\heh}[\nu]$-module by letting the Heisenberg algebra
$\tilde{\heh}[\nu]_{\frac1{r}\Z}$ act trivially. 
We set
\[
V_L^T=S[\nu]\ot U_T
\]
to be a tensor product of $\tilde{\heh}[\nu]$-module on which $\hat{L}_\nu$ acts
by its action on the second component. 
Set $\tilde{\heh}[\nu]=\hat{\heh}[\nu]\oplus\C d$.
$d$ acts on $U_T$ and gives the following grading
\[
d\cdot a=-\frac12\pair{\bar{a},\bar{a}}a\quad\text{for }a\in\hat{L}_\nu.
\]
$V_L^T$ is called the twisted module.

Next we define the $\hat{\nu}$-twisted vertex operator. 
We follow \S2 of \cite{CalLM}.
For each $\alpha\in\heh$ and $m\in\frac1{r}\Z$, let $\alpha_{(j)}$ stand for 
$P_j\alpha\in\heh_{(j)}$. Set 
\begin{align}
\alpha(m)&=\alpha_{(rm)}\ot t^m,\\
E^{\pm}(\alpha,z)&={\rm exp}\left(\sum_{\pm m\in\frac1{r}\Z_+}\frac{\alpha(m)}{m}z^{-m}\right).
\end{align}
Note that from \cite[Proposition 3.4]{LW}, we have the following commutation relation
\begin{equation} \label{comm rel of E}
E^+(\alpha,z_1)E^-(\beta,z_2)=E^-(\beta,z_2)E^+(\alpha,z_1)
\prod_{i=0}^{r-1}\left(1-\eta^i\frac{z_2^{\frac1{r}}}{z_1^{\frac1{r}}}\right)^{\pair{\nu^i\alpha,\beta}}.
\end{equation}
For $a\in\hat{L}_\nu$, as defined in \cite{CalLM,L}, we consider the twisted vertex operator
\begin{equation} \label{tvo}
Y^{\hat{\nu}}(a,z)=r^{-\frac{\pair{\bar{a},\bar{a}}}{2}}\sigma(\bar{a})E^-(-\bar{a},z)E^+(-\bar{a},z)a 
z^{\bar{a}_{(0)}+\frac{\pair{\bar{a}_{(0)},\bar{a}_{(0)}}}2-\frac{\pair{\bar{a},\bar{a}}}2},
\end{equation}
where
\[\sigma(\alpha)=\left\{
\begin{array}{ll}
2^{\pair{\nu\alpha,\alpha}/2}\quad(r=2)\\
(1-\eta^2)^{\pair{\nu\alpha,\alpha}}\quad(r=3)
\end{array}\right.\]
acting on $V_L^T$. (In \cite{CalLM}, this vertex operator is defined for an element
of the lattice vertex operator algebra $V_L$. Since we do not use this fact in this
paper, we simply defined it for $a\in\hat{L}_\nu$.)
We define the component operators $Y^{\hat{\nu}}_{\alpha}(m)$ for $m\in\frac1{r}\Z,\alpha\in L$ by 
\begin{equation}
Y^{\hat{\nu}}(e_\alpha,z)=\sum_{m\in\frac1{r}\Z}Y^{\hat{\nu}}_\alpha(m)z^{-m-\pair{\alpha,\alpha}/2}.
\end{equation}
Set
\begin{equation} \label{rho}
\rho_i=\frac12\pair{(\alpha_i)_{(0)},(\alpha_i)_{(0)}}.
\end{equation}
Then for a simple root $\alpha_i$, we have
\begin{align} \label{1a}
Y^{\hat{\nu}}(\hat{\nu}e_{\alpha_i},z)
=\left.Y^{\hat{\nu}}(e_{\alpha_i},z)\right|_{z^{\frac1{r}}\rightarrow
\eta^{-1}z^{\frac1{r}}}
\end{align}
from \eqref{nu hat 2}. For a component operator, 
\begin{equation}
Y^{\hat{\nu}}_{\nu\alpha_i}(m)=\eta^{rm}Y^{\hat{\nu}}_{\alpha_i}(m).
\end{equation}
From this relation, we know if $\nu\alpha_i=\alpha_i$, then $Y^{\hat{\nu}}_{\alpha_i}(m)=0$ unless $m\in \Z$.
By an explicit calculation, we get, if $n\ne0$,
\begin{equation} \label{[h,Y]}
[h(n),Y^{\hat{\nu}}(e_\alpha,z)]=\pair{h_{(rn)},\alpha_{(-rn)}}z^nY^{\hat{\nu}}(e_\alpha,z).
\end{equation}

\subsection{Twisted affine Lie algebras}
Recall that $\geh$ is a simple Lie algebra of type $A_\ell,D_\ell$ or $E_\ell$. 
Let $\Delta$ be its set of roots.
Then we have the root space decomposition
\[\geh=\heh\oplus\bigoplus_{\alpha\in\Delta}\C x_\alpha.\]
We normalize a root vector $x_\alpha$ so that we have 
\[[x_\alpha,x_\beta]=
\begin{cases}
\ep_{C_0}(\alpha,-\alpha)\alpha&{\rm if}~\alpha+\beta=0\\
\ep_{C_0}(\alpha,\beta)x_{\alpha+\beta}&{\rm if}~\alpha+\beta\in\Delta\\
0&{\rm otherwise}.
\end{cases}\]
Then the symmetric bilinear form $\pair{\cdot,\cdot}$ on $\geh$ reads as
\[\pair{h,x_\alpha}=0,
\qquad\pair{x_\alpha,x_\beta}=
\begin{cases}
\ep_{C_0}(\alpha,-\alpha)&{\rm if}~\alpha+\beta=0,\\
0&{\rm if}~\alpha+\beta\neq0,
\end{cases}\]
where $h\in\heh$.

\begin{table} 
\begin{picture}(400, 80)(-31,-8)
\put(220,20){
\put(40,24){$ D^{(2)}_{l+1}$}
\multiput( 0,0)(20,0){2}{\circle{6}}
\multiput(80,0)(20,0){2}{\circle{6}}
\put(23,0){\line(1,0){14}}
\put(62.5,0){\line(1,0){14}}
\multiput(2.85,-1)(0,2){2}{\line(1,0){14.3}}
\multiput(82.85,-1)(0,2){2}{\line(1,0){14.3}}
\multiput(39,0)(4,0){6}{\line(1,0){2}}
\put(10,0.2){\makebox(0,0){$<$}}
\put(90,0.2){\makebox(0,0){$>$}}
\put(0,-6){\makebox(0,0)[t]{$0$}}
\put(20,-6){\makebox(0,0)[t]{$1$}}
\put(80,-6){\makebox(0,0)[t]{$l\!\! -\!\! 1$}}
\put(100,-6){\makebox(0,0)[t]{$l$}}
}
\put(20,20){
\put(45,24){$ A^{(2)}_{2l-1}$}
\put(6,14){\circle{6}}\put(6,-14){\circle{6}}
\put(20,0){\circle{6}}
\multiput(80,0)(20,0){2}{\circle{6}}
\put(23,0){\line(1,0){14}}
\put(62.5,0){\line(1,0){14}}
\put(18,3){\line(-1,1){9}} \put(18,-3){\line(-1,-1){9}}
\multiput(82.85,-1)(0,2){2}{\line(1,0){14.3}}
\multiput(39,0)(4,0){6}{\line(1,0){2}}
\put(90,0){\makebox(0,0){$<$}}
\put(-2,19){\makebox(0,0)[t]{$0$}}
\put(-2,-11){\makebox(0,0)[t]{$1$}}
\put(20,-6){\makebox(0,0)[t]{$2$}}
\put(80,-6){\makebox(0,0)[t]{$l\!\! -\!\! 1$}}
\put(100,-6){\makebox(0,0)[t]{$l$}}
}
\end{picture}

\begin{picture}(400, 80)(-31,-8)
\put(230,20){
\put(35,24){$D^{(3)}_4$}
\multiput( 20,0)(20,0){3}{\circle{6}}
\put(23,0){\line(1,0){14}}
\put(43,0){\line(1,0){14}}
\multiput(42,-2,6)(0,5,2){2}{\line(1,0){17,0}}
\put(47.5,0){\line(1,-1){5}}
\put(47.5,0){\line(1,1){5}}
\put(20,-6){\makebox(0,0)[t]{$0$}}
\put(40,-6){\makebox(0,0)[t]{$1$}}
\put(60,-6){\makebox(0,0)[t]{$2$}}
}
\put(20,20){
\put(45,24){$E^{(2)}_{6}$}
\multiput(6,0)(20,0){5}{\circle{6}}
\multiput(9,0)(20,0){2}{\line(1,0){14}}
\put(69,0){\line(1,0){14}}
\multiput(48.85,-1)(0,2){2}{\line(1,0){14.3}}
\put(56,0){\makebox(0,0){$<$}}
\put(6,-6){\makebox(0,0)[t]{$0$}}
\put(26,-6){\makebox(0,0)[t]{$1$}}
\put(46,-6){\makebox(0,0)[t]{$2$}}
\put(66,-6){\makebox(0,0)[t]{$3$}}
\put(86,-6){\makebox(0,0)[t]{$4$}}
}
\end{picture}
\caption{Twisted affine Dynkin diagrams \label{tab:Dynkin2}}
\end{table}

Now, we assume $\geh$ to be one of type $A_{2l-1}$, $D_{l+1}$, $E_6$ or $D_4$ and
$\nu$ as in Table \ref{tab:Dynkin1}.
Following \cite{BS,CalLM,L,PS1,PS2}, the automorphism $\nu$ of $\heh$ is lifted to an automorphism $\nu$ of $\geh$ by
\[
\nu x_\alpha=\psi(\alpha)x_{\nu\alpha}
\]
where $\psi$ is some function from $\Delta$ to $\{\pm1\}$. 
For $j\in\Z$ set
\[\geh_{(j)}=\{x\in\geh\mid\nu(x)=\eta^jx\}.\]
The twisted affine Lie algebra $\hat{\geh}[\nu]$ associated to $\geh$ and $\nu$ is given by
\[\hat{\geh}[\nu]=\bigoplus_{m\in\frac1{r}\Z}\geh_{(rm)}\ot t^m\oplus\C c\]
with Lie bracket
\[[x\ot t^m,y\ot t^n]=[x,y]\ot t^{m+n}+\pair{x,y}m\delta_{m+n,0}c,\qquad
[c,\hat{\geh}[\nu]]=0\]
for $m,n\in\frac1{r}\Z$, $x\in\geh_{(rm)}$ and $y\in\geh_{(rn)}$.
We also define the Lie algebra $\tilde{\geh}[\nu]$ by
\[\tilde{\geh}[\nu]=\hat{\geh}[\nu]\oplus\C d\]
where $d$ is the degree operator such that
\[[d,x\ot t^n]=nx\ot t^n\]
for $n\in\frac1{r}\Z$, $x\in\geh_{(rn)}$ and $[d,c]=0$.
This Lie algebra $\hat{\geh}$ (or $\tilde{\geh}$) is isomorphic to a twisted
affine Lie algebra of type $A_{2l-1}^{(2)}$, $D_{l+1}^{(2)}$, $E_6^{(2)}$ or $D_4^{(3)}$ depending on the choice of $L$ and $\nu$. See Table \ref{tab:Dynkin2} for their 
Dynkin diagrams.
\begin{theorem}
$($\cite[Theorem 3.1]{CalLM},\cite[Theorem 3]{FLM},\cite[Theorem 9.1]{L}$)$
The representation of $\hat{\heh}[\nu]$ on $V_L^T$ extends uniquely to a Lie algebra representation of $\tilde{\geh}[\nu]$ on $V_L^T$ such that
\[(x_\alpha)_{(rm)}\ot t^m\mapsto Y^{\hat{\nu}}_\alpha(m)\]
for all $m\in\frac1{r}\Z$ and $\alpha\in L_2=\{\alpha\in L\mid\pair{\alpha,\alpha}=2\}$.
Moreover $V_L^T$ is irreducible as a $\tilde{\geh}[\nu]$-module.
\end{theorem}
In fact, $V_L^T$ is an integrable highest weight module of highest weight 
$\La_0$, where $\Lambda_0$ is the fundamental weight such that $\pair{\La_0,c}=1$
and $\pair{\La_0,\heh_{(0)}}=0=\pair{\La_0,d}$. 
A highest weight vector $1\ot (e_0\ot1)\in V_L^T=S[\nu]\ot U_T$ is denoted by ${\bf 	1}_T$. We also have the following formulas on $V_L^T$.
\begin{align}
&e_\alpha de_\alpha^{-1}=d+\alpha-\frac12\pair{\alpha,\alpha}, \label{ed0}\\
&e_\alpha he_\alpha^{-1}=h-\alpha(h), \label{eh0}\\
&e_\alpha h(j)e_\alpha^{-1}=h(j)\text{ for }j\ne0, \label{ehe0}\\
&e_\alpha Y^{\hat{\nu}}_\beta(j)e_\alpha^{-1}=C(\alpha,\beta)Y^{\hat{\nu}}_\beta(j-\pair{\alpha,\beta_{(0)}}) \label{ex0}.
\end{align}

\subsection{Standard module}

Here we summarize the Cartan subalgebra and simple roots as the twisted 
affine Lie algebras. Set $l=4$ for $E_6^{(2)}$ and $l=2$ for $D_4^{(3)}$.
The Cartan subalgebra is identified with 
\[
\heh_{(0)}\oplus\C c\oplus\C d.
\]
Chevalley generators $h_i$ ($0\le i\le l$) are given by
$h_i=\alpha_i$ when $\nu\alpha_i=\alpha_i$, $h_i=\sum_{p=0}^{r-1}\nu^p\alpha_i$
otherwise for $i\ne0$ and $h_0=-\sum_{p=0}^{r-1}\nu^p\theta^0$ where
$\theta^0$ is given by
\[
\theta^0=\left\{
\begin{array}{ll}
\alpha_1+\cdots+\alpha_{2l-2}&\text{for }A_{2l-1}^{(2)}\\
\alpha_1+\cdots+\alpha_l&\text{for }D_{l+1}^{(2)}\\
\alpha_1+2\alpha_2+2\alpha_3+\alpha_4+\alpha_5+\alpha_6&\text{for }E_6^{(2)}\\
\alpha_1+\alpha_2+\alpha_3&\text{for }D_4^{(3)}.
\end{array}\right.
\]
See \cite[\S8.3]{Kac}. $\tilde{\geh}[\nu]$ contains a finite-dimensional simple 
Lie algebra $\tilde{\geh}[\nu]_{\bar{0}}$ (in Kac's notation) whose Dynkin diagram is 
obtained by removing the node $0$ from that of the twisted affine Lie algebra.
One can take the set of simple roots of $\tilde{\geh}[\nu]_{\bar{0}}$ as that of $\geh$
modulo the automorphism $\nu$, namely, $\{\alpha_1,\ldots,\alpha_l\}$. 
We denote this root lattice of $\tilde{\geh}[\nu]_{\bar{0}}$ by $Q$.

We consider the standard $\tilde{\geh}[\nu]$-module $L(k\La_0)$ of higher level $k$, namely, the integrable highest weight $\tilde{\geh}[\nu]$-module of highest weight $k\La_0$.
Since we know $L(\La_0)\simeq V_L^T$, 
we realize $L(k\La_0)$ as a submodule of the tensor product of $k$ copies of $V_L^T$ 
as \[L(k\La_0)\simeq U(\tilde{\geh}[\nu])\cdot v_0 \subset {(V_L^T)}^{\ot k},\]
where $v_0={\bf 1}_T\ot\cdots\ot{\bf 1}_T$ is a highest weight vector of $L(k\La_0)$. 
On $L(k\La_0)$, elements of $\tilde{\geh}[\nu]$ act through the coproduct
\[
\Delta^{(k-1)}(x)=x\ot1\ot\cdots\ot1+1\ot x\ot\cdots\ot1+\cdots+1\ot1\ot\cdots\ot x,
\]
where there are $k$ components in each term. It is also true for the twisted vertex 
operator $Y^{\hat{\nu}}(e_\alpha,z)$. For a simple root $\alpha_i$ and a positive
integer $n$, we set 
\begin{equation}\label{tvo2}
x^{\hat{\nu}}_{n\alpha_i}(z)=x^{\hat{\nu}}_{\alpha_i}(z)^n=[\Delta^{(k-1)}
(Y^{\hat{\nu}}(e_{\alpha_i},z))]^n.
\end{equation}
Note that $x^{\hat{\nu}}_{(k+1)\alpha_i}(z)=0$.
We also define a component operator $x^{\hat{\nu}}_{n\alpha_i}(m)$ by
\[
x^{\hat{\nu}}_{n\alpha_i}(z)=\sum_{m\in\frac1{r}\Z}x^{\hat{\nu}}_{n\alpha_i}(m)z^{-m-n}.
\]
Later, we will use the following commutation relation which can be shown using \eqref{[h,Y]}.  If $m\ne0$, 
\begin{equation} \label{[h,x]}
[h(m),x^{\hat{\nu}}_{n\alpha_i}(z)]=n\pair{h_{(rm)},\alpha_{(-rm)}}z^mx^{\hat{\nu}}_{n\alpha_i}(z),
\end{equation}
or equivalently,
\[
[h(m),x^{\hat{\nu}}_{n\alpha_i}(j)]=n\pair{h_{(rm)},\alpha_{(-rm)}}x^{\hat{\nu}}_{n\alpha_i}(j+m).
\]

For $\alpha\in L$, $E^\pm(\alpha,z)$ acts on $(V_L^T)^{\ot k}$ diagonally and hence also on $L(k\La_0)$. 
One checks 
\begin{align}
\label{Ehcom}
[E^\pm(\alpha,z),h(n)]=\left\{
\begin{array}{ll}
k\pair{\alpha_{(-rn)},h_{(rn)}}z^nE^\pm(\alpha,z)&\text{ if }\mp n>0,\\
0&\text{ othewise}.
\end{array}\right.
\end{align}
$e_\alpha$ also acts on $(V_L^T)^{\ot k}$ diagonally, namely, 
\[
e_\alpha\mapsto e_\alpha\ot\cdots\ot e_\alpha.
\]
This corresponds to the translation operator of the
affine Weyl group of $\tilde{\geh}[\nu]$. See Section 1.5 of \cite{BKP} for the 
untwisted case.
Based on the calculations \eqref{ed0}-\eqref{ex0} on $V_L^T$, we define the adjoint action on $\tilde{\geh}[\nu]$ of the multiplicative group
isomorphic to $Q$ by
\begin{align}
\label{ec}&e_\alpha ce_\alpha^{-1}=c,\\
\label{ed}&e_\alpha de_\alpha^{-1}=d+\alpha-\frac12\pair{\alpha,\alpha}c,\\
\label{eh}&e_\alpha he_\alpha^{-1}=h-\alpha(h)c\text{ for }h\in\heh_{(0)},\\
\label{ehe}&e_\alpha h(j)e_\alpha^{-1}=h(j)\text{ for }j\ne0,\\
\label{ex}&e_\alpha x^{\hat{\nu}}_\beta(j)e_\alpha^{-1}=C(\alpha,\beta)x^{\hat{\nu}}_\beta(j-\pair{\alpha,\beta_{(0)}}).
\end{align}
Note that $c=k$ on $L(k\La_0)$.

Next we state the vertex operator formula that will be used later. This is a 
twisted version of (1.27) in \cite{BKP}. The proof is completely parallel to 
\cite[Theorem 5.6]{LP} or \cite[Theorem 6.4]{P}. 

\begin{lemma}
For a simple root $\alpha_i$, renormalize the twisted vertex operator $x^{\hat{\nu}}_{\alpha_i}(z)$ as $\tilde{x}^{\hat{\nu}}_{\alpha_i}(z)
=r\sigma(\alpha_i)^{-1}x^{\hat{\nu}}_{\alpha_i}(z)$. Then, for $p,q\geq0$ such that $p+q=k$, we have
\begin{align}\label{voformula}
\frac{1}{p!}E^-(\alpha_i,z)(z\tilde{x}_{\alpha_i}^{\hat{\nu}}(z))^pE^+(\alpha_i,z)=\frac{1}{q!}\ep_C(\alpha_i,-\alpha_i)^{-q}(z\tilde{x}_{-\alpha_i}^{\hat{\nu}}(z))^qe_{\alpha_i}z^{(\alpha_i)_{(0)}+k\rho_i}
\end{align}
as an operator on $(V_L^T)^{\ot k}$ or $L(k\La_0)$.
Furthermore, (\ref{voformula}) can be rewritten as 
\begin{equation}\label{voformula2}
E^-(\alpha_i,z){\rm exp}(z\tilde{x}_{\alpha_i}^{\hat{\nu}}(z))E^+(\alpha_i,z)
={\rm exp}(\ep_C(\alpha_i,-\alpha_i)^{-1}z\tilde{x}_{-\alpha_i}^{\hat{\nu}}(z))e_{\alpha_i}z^{(\alpha_i)_{(0)}+k\rho_i}.
\end{equation}
\end{lemma}
\begin{proof}
Set $y_{\alpha}(z)=r^{\frac{\pair{\alpha,\alpha}}2}\sigma(\alpha)^{-1}Y^{\hat{\nu}}(e_{\alpha_i},z)$. From (\ref{tvo}), we have
\[E^-(\alpha_i,z)y_{\alpha_i}(z)E^+(\alpha_i,z)=e_{\alpha_i}z^{(\alpha_i)_{(0)}+\rho_i-1}\]
on $V_L^T$. Since $y_{\alpha_i}(z)^2=0$ on $V_L^T$, 
$\tilde{x}_{\alpha_i}^{\hat{\nu}}(z)^p$ acts on $(V_L^T)^{\ot k}$ as 
$p!\sum y_{\alpha_i}(z)^{j_1}\ot\cdots\ot y_{\alpha_i}(z)^{j_k}$, where $j_m\in\{0,1\}$ and $j_1+\cdots+j_k=p$. Since $E^\pm(\alpha_i,z),
e_{\alpha_i},z^{(\alpha_i)_{(0)}+\rho_i}$ all act grouplike on $(V_L^T)^{\ot k}$, 
\begin{align*}
\text{LHS of \eqref{voformula}}
&=z^pE^-(\alpha_i,z)(\sum y_{\alpha_i}(z)^{j_1}\ot\cdots\ot y_{\alpha_i}(z)^{j_k})E^+(\alpha_i,z)\\
&=\ep_C(\alpha_i,-\alpha_i)^{-q}z^q(\sum y_{-\alpha_i}(z)^{1-j_1}e_{\alpha_i}z^{(\alpha_i)_{(0)}+\rho_i}\ot\cdots\ot y_{-\alpha_i}(z)^{1-j_k}e_{\alpha_i}z^{(\alpha_i)_{(0)}+\rho_i})\\
&=\frac1{q!}\ep_C(\alpha_i,-\alpha_i)^{-q}(z\tilde{x}_{-\alpha_i}^{\hat{\nu}}(z))^qe_{\alpha_i}z^{(\alpha_i)_{(0)}+k\rho_i}.
\end{align*}
Here we have used $z^{(\alpha_i)_{(0)}-\rho_i}e_{-\alpha_i}^{-1}=\ep_C(\alpha_i,-\alpha_i)^{-1}e_{\alpha_i}
z^{(\alpha_i)_{(0)}+\rho_i}$.

Noting $\tilde{x}^{\hat{\nu}}_{\alpha_i}(z)^{k+1}=0$ on $(V_L^T)^{\ot k}$, we obtain
(\ref{voformula2}) from (\ref{voformula}) immediately.
\end{proof}

\subsection{Principal subspace}

We will introduce the notion of the principal subspace of $L(k\La_0)$ and twisted quasi-particle bases which we will use to construct parafermionic bases.
First, denote by $\Delta_+$ the set of positive roots and by
\[\neh=\bigoplus_{\alpha\in\Delta_+}\C x_\alpha\]
the Lie subalgebra of $\geh$ which is the nilradical of a Borel subalgebra. Consider its twisted affinization
\[\hat{\neh}[\nu]=\bigoplus_{m\in\frac1{r}\Z}\neh_{(rm)}\ot t^m\oplus\C c\]
and its subalgebra
\[
\bar{\neh}[\nu]=\bigoplus_{m\in\frac1{r}\Z}\neh_{(rm)}\ot t^m.
\]
In \cite{BS,FS1,FS2,PS1,PS2}, the principal subspace $W(k\La_0)$ of $L(k\La_0)$ is defined as
\[W(k\La_0)=U(\bar{\neh}[\nu])\cdot v_0.\]

From \cite{BS}, we review twisted quasi-particle monomials.
We define the twisted quasi-particle of color $i$, charge $n$ and energy $-m$ for each simple root $\alpha_i$, $n\in{\mathbb N}$, and $m\in\frac1{r}\Z$ as the coefficient $x^{\hat{\nu}}_{n\alpha_i}(m)$ in (\ref{tvo2}).
A twisted quasi-particle monomial is then defined by
\begin{equation} \label{b}
b=x^{\hat{\nu}}_{n_{r_l^{(1)},l}\alpha_l}(m_{r_l^{(1)},l})\cdots x^{\hat{\nu}}_{n_{1,l}\alpha_l}(m_{1,l})\cdots x^{\hat{\nu}}_{n_{r_1^{(1)},1}\alpha_1}(m_{r_1^{(1)},1})\cdots x^{\hat{\nu}}_{n_{1,1}\alpha_1}(m_{1,1}).
\end{equation}
The sequence 
\[{\mathcal R}^{\prime}=\left(n_{r_l^{(1)},l},\ldots,n_{1,l};\ldots;n_{r_1^{(1)},1}\ldots n_{1,1}\right)\]
is called its $charge$-$type$. We assume $1\le n_{r_i^{(1)},i}\leq\ldots\leq n_{1,i}\le k$ for each $i$.
The $dual$-$charge$-$type$
\[{\mathcal R}=\left(r_l^{(1)},\ldots, r_l^{(k)};\ldots;r_1^{(1)},\ldots,r_1^{(k)}\right)\]
is defined in the way that $(r_i^{(1)},\ldots,r_i^{(k)})$ is the transposed partition of $(n_{1,i},\ldots,n_{r_i^{(1)},i})$ for each $i$.
Namely, $r_i^{(s)}$ stands for the number of quasi-particles of color $i$ and charge greater than or equal to $s$ in the monomial $b$. So 
we have $r_i^{(1)}\geq r_i^{(2)}\geq\cdots\geq r_i^{(k)}\geq0$. 
The $color$-$type$ is defined by
\[\mathcal{C}=(r_l,\ldots,r_1)\]
where
\begin{equation} \label{col type}
r_i=\sum_{p=1}^{r_i^{(1)}}n_{p,i}=\sum_{s=1}^k r_i^{(s)}.
\end{equation}

According to \cite{BS}, we consider the following conditions ($C1$)-($C3$) for the mode $m$ in 
$x^{\hat{\nu}}_{n\alpha_i}(m)$ in $b$.
\begin{align*}
(C1)\quad&m_{p,i}\in\rho_i\Z\quad\text{for }1\leq p\leq r_i^{(1)},1\leq i\leq l,\\
(C2)\quad&m_{p,i}\leq-(2p-1)\rho_in_{p,i}-\pair{(\alpha_i)_{(0)},(\alpha_{i-1})_{(0)}}\sum_{q=1}^{r_{i-1}^{(1)}}{\rm min}\{n_{p,i},n_{q,i-1}\}
\quad\text{for }1\leq p\leq r_i^{(1)},~1\leq i\leq l,\\
(C3)\quad&m_{p+1,i}\leq m_{p,i}-2\rho_in_{p,i}~{\rm if}~n_{p+1,i}=n_{p,i}\quad\text{for }1\leq p\leq r_i^{(1)}-1,~1\leq i\leq l.
\end{align*}
Here we understand $r_0^{(1)}=0$.
Set
\[
B_W=\bigcup_{\substack{r_1^{(1)}\geq\cdots\geq r_1^{(k)}\geq0\\
\vdots\\
r_l^{(1)}\geq\cdots\geq r_l^{(k)}\geq0}}
\{b\text{ as in \eqref{b}}\mid b\text{ satisfies }(C1),(C2)\text{ and }(C3)\}
\]
We know that the principal subspace $W(k\La_0)$ has a basis consisting from 
twisted quasi-particle monomials.
\begin{theorem} \label{th:BS}
$($\cite[Theorem 5.1]{BS}$)$ 
The set $\mathcal{B}_W=\{bv_0\mid b\in B_W\}$
is a basis of the principal subspace $W(k\La_0)$.
\end{theorem}

\section{Quasi-particle bases of standard modules}

\subsection{Spanning sets for standard modules}

\begin{lemma} \label{spanning set}
{\phantom{a}}
\begin{enumerate}
\item $L(k\La_0)=U(\hat{\heh}[\nu]^-)QW(k\La_0)$
\item $L(k\La_0)=QW(k\La_0)$
\end{enumerate}
\end{lemma}
\begin{proof}
$\hat{\geh}[\nu]$ is generated by $(x_\beta)_{(rm)}\ot t^m$ acting as $x^{\hat{\nu}}_{\beta}(m)$ on $L(k\La_0)$ for $m\in\frac1{r}\Z$, $\beta\in\Delta$.
Since every $x_\beta$ can be expressed by taking brackets with $x_{\alpha_i}$ for $1\leq i\leq\ell$ and $x^{\hat{\nu}}_{\nu\alpha_i}(m)$ can also be expressed as $x^{\hat{\nu}}_{\alpha_i}(m)$ by using (\ref{1a}), the standard module $L(k\La_0)$ is spanned by noncommutative monomials in $x^{\hat{\nu}}_{\pm\alpha_i}(m)$, $i=1,\cdots,l$, $m\in\frac1{r}\Z$.
By using the vertex operator formula (\ref{voformula}), we can express $x^{\hat{\nu}}_{-\alpha_i}(m)$ in terms of $x^{\hat{\nu}}_{\alpha_i}(m^\prime)$, $e_{\alpha_i}$ and a polynomial in $U(\hat{\heh}[\nu])$.
From \eqref{[h,x]} and \eqref{ehe}, we can move the elements of $\hat{\heh}[\nu]^-$ to the left and the elements of $\hat{\heh}[\nu]^+$ to the right.
Since $h(n)v_0=0$ for $n>0$, we see that (1) holds.
$U(\hat{\heh}[\nu]^-)$ is spanned by the coefficients of $E^-(-\alpha_i,z)$.
Therefore we obtain (2) from (1) by using the vertex operator formula (\ref{voformula}) for $q=0$ and the commutation relation (\ref{Ehcom}).
\end{proof}
The second statement in Lemma 5 implies that
\[\{e_\mu v\mid\mu\in Q, v\in \mathcal{B}_W\}\]
spans $L(k\La_0)$.
But this is not a basis. 

For the proofs of Proposition \ref{prop:span} and the main theorem, 
we introduce a linear order on the quasi-particle monomials in $B_W$ following \cite{BS}.
For two monomials $b$ abd $\bar{b}$ with charge-types $\mathcal{R}'$ and 
$\bar{\mathcal{R}}'=(\bar{n}_{\bar{r}_l^{(1)},l},\ldots,\bar{n}_{1,l};\ldots;\bar{n}_{\bar{r}_1^{(1)},1},\ldots,\bar{n}_{1,1})$ and
with energies $(m_{r_l^{(1)},l},\ldots,m_{1,1})$ and $(\bar{m}_{\bar{r}_l^{(1)},l},\ldots,\bar{m}_{1,1})$, we write $b<\bar{b}$ if one 
of the follwoing conditions holds
\begin{enumerate}
\item $\mathcal{R}'<\bar{\mathcal{R}}'$
\item $\mathcal{R}'=\bar{\mathcal{R}}'$ and $(m_{r_l^{(1)},l},\ldots,m_{1,1})<(\bar{m}_{\bar{r}_l^{(1)},l},\ldots,\bar{m}_{1,1})$
\end{enumerate}
where we write $\mathcal{R}'<\bar{\mathcal{R}}'$ if there exists $i$ and $s$ such
that $r_j^{(1)}=\bar{r}_j^{(1)},n_{t,j}=\bar{n}_{t,j}$ for $j<i,1\le t\le r_j^{(1)}$ and 
$n_{1,i}=\bar{n}_{1,i},n_{2,i}=\bar{n}_{2,i},\ldots,n_{s-1,i}=\bar{n}_{s-1,i},n_{s,i}<\bar{n}_{s,i}$
or $n_{t,i}=\bar{n}_{t,i}$ for $1\le t\le r_i^{(1)}$, $r_i^{(1)}<\bar{r}_i^{(1)}$. In the case that
$\mathcal{R}=\bar{\mathcal{R}}'$, we apply this definition to the sequences of energies
to similarly define $(m_{r_l^{(1)},l},\ldots,m_{1,1})<(\bar{m}_{\bar{r}_l^{(1)},l},\ldots,\bar{m}_{1,1})$.

Set 
\[
B_H=\left\{h_{\alpha_l}\cdots h_{\alpha_1}\left|
\begin{array}{l}
h_{\alpha_i}=\alpha_i(-m_{t_i,i})^{n_{t_i,i}}\cdots\alpha_i(-m_{1,i})^{n_{1,i}},i=1,\ldots l,\\
t_i\in\Z_{\ge0},m_{t_i,i}>\cdots>m_{1,i},m_{p,i}\in\rho_i\mathbb{N},n_{p,i}\in\mathbb{N}
\end{array}\right.\right\}
\]
and $B'_W=B_W\cap M'_{QP}$.
The following proposition can be proved in the same way as Lemma 2.3 in \cite{BKP}.
\begin{proposition} \label{prop:span}
The set $\mathcal{B}_L=\{e_\mu hbv_0\mid \mu\in Q,~h\in B_H,~b\in B'_W\}$ spans $L(k\La_0)$.
\end{proposition}
\begin{proof}
By Lemma \ref{spanning set} (1), the set of vectors
\[\{e_\mu hbv_0\mid \mu\in Q,h\in B_H,b\in B_W\}\]
spans $L(k\La_0)$.
It suffices to check that the arguments used in the proof of \cite[Lemma 2.3]{BKP}
also hold for our case.
By the vertex operator formula (\ref{voformula}), a quasi-particle $x^{\hat{\nu}}_{k\alpha_i}(m)$ is expressed by $e_{\alpha_i}$ and monomials in $U(\hat{\heh}[\nu]^\pm)$.
Then we can move $e_\alpha$ and elements of $\hat{\heh}[\nu]^-$ to the left, and  elements of $\hat{\heh}[\nu]^+$ to the right by the relations (\ref{Ehcom}) and \eqref{ehe}.
As a result, we can express $e_\mu hbv_0$ as a linear combination of vectors $e_{\mu^\prime}h^\prime b^\prime v_0$, where $\mu^\prime\in Q, h^\prime\in B_H$, and $b^\prime\in M'_{QP}$.
Note that $b^\prime$ contains no quasi-particles $x^{\hat{\nu}}_{k\alpha_i}(m)$, but 
is not necessarily in $B_W$.
Take any vector $e_{\mu^\prime}h^\prime b^\prime v_0$ that is not in $B_W$, if it exists. 
Since $b'v_0\in W(k\La_0)$, it is expressed as a linear combination of vectors
$b''v_0$ such that $b''\in B_W$ by Theorem \ref{th:BS}. $b''$ may contain
$x^{\hat{\nu}}_{k\alpha_i}(m)$, but it is greater than $b$ with respect to the order "$<$" 
defined above. Since the set $B_W$ is upper bounded with respect to this order,
the process of eliminating quasi-particles $x^{\hat{\nu}}_{k\alpha_i}(m)$ end in finitely
many steps and $b$ in any term in the final linear combination belongs to $B_W\cap
M'_{QP}$.
\end{proof}

\subsection{The main theorem}
Consider the decomposition
\[L(k\La_0)=\bigoplus_{s\in\Z}L(k\La_0)_s,\quad {\rm where}\quad L(k\La_0)_s=\bigoplus_{s_2,\cdots,s_l\in\Z}L(k\La_0)_{s_l\alpha_l+\cdots+s_2\alpha_2+s\alpha_1}.\]
To prove our main theorem, we use the Georgiev-type projection such that
\[\pi_{\mathcal{R}_{\alpha_1}}:L(k\La_0)\rightarrow L(\La_0)_{r_1^{(1)}}\ot\cdots\ot L(\La_0)_{r_1^{(k)}},\]
where $\mathcal{R}_{\alpha_1}=(r_1^{(1)},r_1^{(2)},\ldots,r_1^{(k)})$ is a fixed dual-charge-type for the color 1 and $r_1=\sum_{s=1}^kr_1^{(s)}$.
This projection is naturally generalized to $L(k\La_0)[[w_{t_l,l},\ldots,w_{1,1},z_{r_l^{(1)},l},\ldots,z_{1,1}]]$.
We also denote this generalization by $\pi_{\mathcal{R}_{\alpha_1}}$.
Set $\alpha_i(z)_-=\sum_{m<0}\alpha_i(m)z^{-m-1}$.
We consider the vector
\[e_\mu\alpha_l(-m_{t_l,l}^\prime)^{n_{t_l,l}^\prime}\cdots\alpha_1(-m_{1,1}^\prime)^{n_{1,1}^\prime}x^{\hat{\nu}}_{n_{r_l^{(1)},l}\alpha_l}(m_{r_l^{(1)},l})\cdots x^{\hat{\nu}}_{n_{1,1}\alpha_1}(m_{1,1})v_0\]
with dual-charge-type $\mathcal{R}=(\mathcal{R}_{\alpha_l},\cdots,\mathcal{R}_{\alpha_1})$.
Recall that the image of this vector with respect to $\pi_{\mathcal{R}_{\alpha_1}}$ coincides with the coefficient of the corresponding projection of the generating function
\[e_\mu\alpha_l(w_{t_l,l})_-^{n_{t_l,l}^\prime}\alpha_1(w_{1,1})_-^{n_{1,1}^\prime}x^{\hat{\nu}}_{n_{r_l^{(1)},l}\alpha_l}(z_{r_l^{(1)},l})\cdots x^{\hat{\nu}}_{n_{1,1}\alpha_1}(z_{1,1})v_0.\]

In order to prove the main theorem, we need a generalization of the twisted vertex
operator $Y^{\hat{\nu}}(a,z)$ defined in section \ref{subsec:TVO} to the case where
$a$ belongs to an extension of the weight lattice $P$ of $\geh$. We do not repeat 
its definition. See \cite{DL}. 

\begin{proposition} \label{prop:DL}
Let $a,b$ be elements of an extension $\hat{P}_\nu$ of $P$ such that
$\pair{\bar{a},\bar{b}}\in\Z$. Then we have the following commutation relation for the 
twisted vertex operators.
\[
Y^{\hat{\nu}}(a,z_1)Y^{\hat{\nu}}(b,z_2)=(-1)^{\pair{\bar{a},\bar{b}}}c_\nu(\bar{a},\bar{b})
Y^{\hat{\nu}}(b,z_2)Y^{\hat{\nu}}(a,z_1)
\]
Here $c_\nu(\bar{a},\bar{b})$ is some constant which belongs to $\C^\times$.
\end{proposition}
\begin{proof}
We use \cite[Theorem 5.2]{DL} in the case when $\mathbf{h}_*=0$. The condition 
$\pair{\bar{a},\bar{b}}\in\Z$ is important, since we use the property 
$z^m\delta(z)=\delta(z)$ ($m\in\Z$) for the formal delta function.
\end{proof}

Let $\lambda_i$ ($i=1,\ldots,\ell$) be the fundamental weights of $\geh$ and 
set
\[
Y^{\hat{\nu}}(e_{\lambda_i},z)=\sum_{m\in\frac1{r}\Z}A_{\lambda_i}(m)z^{-m+\pair{(\lambda_i)_{(0)},(\lambda_i)_{(0)}}/2
-\pair{\lambda_i,\lambda_i}/2}.
\]
From Proposition \ref{prop:DL}, we have
\begin{equation} \label{Ax}
A_{\lambda_i}(m)Y_{\alpha_j}^{\hat{\nu}}(n)=(-1)^{\delta_{ij}}c_\nu(\lambda_i,\alpha_j)Y_{\alpha_j}^{\hat{\nu}}(n)A_{\lambda_i}(m)
\end{equation}
on $V_P^T$, which is the extended space of $V_L^T$ by
enlarging the root lattice $L$ to the weight lattice $P$. We also have
\begin{equation} \label{[h,A]}
[h(n),A_{\lambda_i}(m)]=\pair{h_{rn},(\lambda_i)_{(-rn)}}A_{\lambda_i}(m+n).
\end{equation}
Moreover, 
\begin{equation} \label{A1}
A_{\lambda_i}(m)\mathbf{1}_T\in U(\heh[\nu]^-)e_{\lambda_i}\text{ for }m\ge0
\text{ and }A_{\lambda_i}(0)\mathbf{1}_T=e_{\lambda_i}.
\end{equation}

\begin{theorem} \label{th:main1}
The set $\mathcal{B}_L$ is a basis of $L(k\La_0)$.
\end{theorem}
\begin{proof}
We should prove the linear independence of $\mathcal{B}_L$.
We consider a linear combination of vectors in $\mathcal{B}_L$,
\begin{align}\label{fc}
\sum c_{\mu,h,b}e_\mu hbv_0=0
\end{align}
of the fixed degree and $\heh_{(0)}$-weight.
From \eqref{eh},  for a $\heh_{(0)}$-weight $\rho$, the action of $e_\mu$ maps 
the weigt space $V_\rho$ to $V_{\rho+k\mu}$.
Hence, we may assume that a summand in \eqref{fc} with the maximal charge of color 1, 
${\rm chg}_1b$, has $\mu$ with $\alpha_1$ coordinate zero.
Namely, we assume that summands appear in the form
\begin{itemize}
\item[(A)] $e_\mu hbv_0$ with ${\rm chg}_1b=r_1$ and $\mu=c_l\alpha_l+\cdots+c_2\alpha_2,\quad {\rm or}$
\item[(B)] $e_{\bar{\mu}}\bar{h}\bar{b}v_0$ with ${\rm chg}_1\bar{b}<r_1$ and $\bar{\mu}=\bar{c}_l\alpha_l+\cdots+\bar{c}_1\alpha_1,\quad{\rm where}~\bar{c}_1>0$.
\end{itemize}
Among the vectors $v=e_\mu hbv_0$ with ${\rm chg}_1v=r_1$, we choose a vector with the maximal charge-type $\mathcal{R}_{\alpha_1}^\prime$ and the corresponding dual-charge-type
\[\mathcal{R}_{\alpha_1}=(r_1^{(1)},\cdots,r_1^{(k-1)})\]
for the color $i=1$ where $r_1=r_1^{(1)}+\cdots+r_1^{(k-1)}$.
Note that $r_1^{(k)}=0$ for a vector in $\mathcal{B}_L$.
Denote the Georgiev-type projection by $\pi_{\mathcal{R}_{\alpha_1}}$.
Since
\[e_{\alpha_1}({\bf1}_T \ot\cdots\ot{\bf1}_T)=e_{\alpha_1}{\bf1}_T\ot\cdots\ot e_{\alpha_1}{\bf1}_T,\]
we have
\[e_{\bar{\mu}}\bar{h}\bar{v}\in\bigoplus_{\substack{s_1,\cdots,s_{k-1}\in\Z\\
s_k>0}}L(\La_0)_{s_1}\ot\cdots\ot L(\La_0)_{s_k}.\]
Therefore, for the vectors of the form (B) we have $\pi_{\mathcal{R}_{\alpha_1}}(e_{\bar{\mu}}\bar{h}\bar{b})v_0=0$.
This means that the $\pi_{\mathcal{R}_{\alpha_1}}$ projection of the sum (\ref{fc}) contains only the summands of the form (A). Applying the same trick for the simple roots 
$\alpha_2,\ldots,\alpha_l$, we can assume $\mu=0$ in \eqref{fc}.

Consider a linear combination
\begin{align}\label{fc2}
c_{h,b}hbv_0+\sum_{b'>b}c_{h',b'}h'b'v_0=0.
\end{align}
For a monomial $h=h_l(-m_{t_l,l})^{n_{t_l,l}}\cdots h_1(-m_{1,1})^{n_{1,1}}$, set
$\bar{h}=h_l(m_{t_l,l})^{n_{t_l,l}}\cdots h_1(m_{1,1})^{n_{1,1}}$ and multiply $\bar{h}$
from left to \eqref{fc2}. Using \eqref{[h,x]}, it turns out to be
\[
c_{h,b}bv_0+\sum_{b''>b}c'_{h'',b''}h''b''v_0=0
\]
up to an overall scalar multiple.
Now, using \eqref{Ax},\eqref{[h,A]} and \eqref{A1}, one can argue in a similar way to the proof of
Theorem 5.2 of \cite{G1} to prove $c_{h,b}=0$. 
For the commutation relation of $e_{\lambda_i}$ and $Y^{\hat{\nu}}_{\alpha_j}(m)$ is
given by substituting $\alpha=\lambda_i$ in \eqref{ex0}.
In place of (5.25) of \cite{G1}, we use
\[
\pi_{\mathcal{R}_{\alpha_1}}b'x^{\hat{\nu}}_{s\alpha_i}(-s\rho_i)v_0
=\text{const }\pi_{\mathcal{R}_{\alpha_1}}b'
\left(1\ot\cdots\ot1\ot\underbrace{e_{\alpha_1}\ot\cdots\ot e_{\alpha_1}}_{s\text{ factors}}\right)v_0.
\]
Once we show $c_{h,b}=0$, substitute this relation to \eqref{fc2} and
continue the process, then one eventually shows all coefficients are zero.
%
\begin{comment}
Therefore we reduce (\ref{fc}) to a linear combination of vector $v=e_\mu hbv_0$ such that ${\rm chg}_1v=0$ by using the intertwining operator $A_{\lambda_1}$ and the injectivity of $e_\mu$ (see \cite[Section 3.2]{B}).
Note that the Georgiev-type projection $\pi_{\mathcal{R}_{\alpha_1}}$ and the action $e_\mu$ change the energies of $v$ for some $\alpha_2,\cdots,\alpha_l$, but their dual-charge-types $\mathcal{R}_{\alpha_2},\cdots,\mathcal{R}_{\alpha_l}$ are not changed.
Furthermore the vectors obtained by this argument satisfy the conditions of $\mathcal{B}_L$.

Then we start with a similar argument for the color $\alpha_2$ by choosing the monomials with the maximal 2-charge and corresponding Georgiev-type projection.
Finally we get $c_{\mu,h,b}=0$ for all coefficients by continuing the same procedure for $\alpha_3,\cdots,\alpha_l$.
\end{comment}
%
\end{proof}

\section{Parafermionic bases}

\subsection{Vacuum space and twisted $\mathcal{Z}$-algebra}
Denote by ${L(k\La_0)}^{\hat{\heh}[\nu]^+}$ the vacuum space of the standard module $L(k\La_0)$ , i.e.
\begin{align}
\label{vacuumsp}
{L(k\La_0)}^{\hat{\heh}[\nu]^+} = \{v\in L(k\La_0)\mid \hat{\heh}[\nu]^+\cdot v=0\}.
\end{align}
By the Lepowsky-Wilson theorem \cite{LW} (A5.3) we have the canonical isomorphism of $d$-graded linear spaces
\begin{align}
\label{LWthm}
U(\hat{\heh}[\nu]^-)\ot L(k\La_0)^{\hat{\heh}[\nu]^+}\xrightarrow{\simeq} L(k\La_0)
\end{align}
\[h\ot u \mapsto h\cdot u\]
where
$U(\hat{\heh}[\nu]^-)\simeq S(\hat{\heh}[\nu]^-)$
is the Fock space of level $k$ for the Heisenberg subalgebra $\hat{\heh}[\nu]_{\frac1{r}\Z}$ with the action of $c$ being the multiplication by scalar $k$.
We consider the projection
\[\pi^{\hat{\heh}[\nu]^+}:L(k\La_0)\rightarrow {L(k\La_0)}^{\hat{\heh}[\nu]^+}\]
given by the direct decomposition
\begin{align}
\label{decomposition}
L(k\La_0) = {L(k\La_0)}^{\hat{\heh}[\nu]^+}\oplus \hat{\heh}[\nu]^-U(\hat{\heh}[\nu]^-)\cdot {L(k\La_0)}^{\hat{\heh}[\nu]^+}.
\end{align}
By (\ref{ehe}), we have the projective representation of {\it Q} on the vacuum space $L(k\La_0)^{\hat{\heh}[\nu]^+}$.

We set
\[\mathcal{Z}_{n\alpha}(z)=E^-(\alpha,z)^{n/k}x_{n\alpha}^{\hat{\nu}}(z)E^+(\alpha,z)^{n/k}\]
for a quasi-particle of charge {\it n} and a root $\alpha$.
It is called the $\mathcal{Z}$-operator.
Note that the action of $\mathcal{Z}$-operators commutes with 
the action of the Heisenberg algebra $\hat{\heh}[\nu]_{\frac1{r}\Z}$ on the standard module $L(k\La_0)$.
More generally, we need to define the $\mathcal{Z}$-operators for quasi-particles of charge-type $\mathcal{R}^\prime=(n_{r_l^{(1)},l},\cdots,n_{1,1})$.
For $x_{\mathcal{R}^\prime}^{\hat{\nu}}(z_{r_l^{(1)},l},\cdots,z_{1,1})=x_{n_{r_l^{(1)},l}\alpha_l}^{\hat{\nu}}(z_{r_l^{(1},l})\cdots x_{n_{1,1}\alpha_1}^{\hat{\nu}}(z_{1,1})$ of charge-type $\mathcal{R}^\prime$, we define
\begin{align}
\mathcal{Z}_{\mathcal{R}^\prime}(z_{r_l^{(1)},l},\cdots,z_{1,1})=&E^-(\alpha_l,z_{r_l^{(1)},l})^{n_{r_l^{(1)},l}/k}\cdots E^-(\alpha_1,z_{1,1})^{n_{1,1}/k}x_{\mathcal{R}^\prime}^{\hat{\nu}}(z_{r_l^{(1)},l},\cdots,z_{1,1})\nonumber \\
&\times E^+(\alpha_l,z_{r_l^{(1)},l})^{n_{r_l^{(1)},l}/k}\cdots E^+(\alpha_1,z_{1,1})^{n_{1,1}/k}.
\label{zo}
\end{align}
For convenience, we write this formal Laurent series by
\[
\mathcal{Z}_{\mathcal{R}^\prime}(z_{r_l^{(1)},l},\cdots,z_{1,1})=\sum_{m_{r_l^{(1)},l},\cdots,m_{1,1}\in\frac1{r}\Z}\mathcal{Z}_{\mathcal{R}^\prime}(m_{r_l^{(1)},l},\cdots,m_{1,1})z_{r_l^{(1)},l}^{-m_{r_l^{(1)},l}-n_{r_l^{(1)},l}}\cdots z_{1,1}^{-m_{1,1}-n_{1,1}}.
\]
Since $\mathcal{Z}$-operators act on the vacuum space and we can express quasi-particle monomials in terms of $\mathcal{Z}$-operators by reversing (\ref{zo}), we have
\[\pi^{\hat{\heh}[\nu]^+}:x_{\mathcal{R}^\prime}(z_{r_l^{(1)},l},\cdots,z_{1,1})v_0\mapsto\mathcal{Z}_{\mathcal{R}^\prime}(z_{r_l^{(1)},l},\cdots,z_{1,1})v_0.\]

Now, Theorem \ref{th:main1} implies
\begin{theorem} \label{th:main2}
The set of vectors
\[e_\mu\mathcal{Z}_{\mathcal{R}^\prime}(m_{r_l^{(1)},l},\cdots,m_{1,1})v_0\]
such that $\mu\in Q$ and the charge-type $\mathcal{R}^\prime$ and the energy-type
$(m_{r_l^{(1)},l},\cdots,m_{1,1})$ satisfy the conditions for
$B'_W$ is a basis of the vacuum space $L(k\La_0)^{\hat{\heh}[\nu]^+}$.
\end{theorem}
\begin{comment}
\begin{proof}
The image of elements of the basis $\mathcal{B}_L$ with respect to the projection $\pi^{\hat{\heh}[\nu]^+}$ span the vacuum space $L(k\La_0)^{\hat{\heh}[\nu]^+}$.
Since the projection $\pi^{\hat{\heh}[\nu]^+}$ annihilates all vectors $e_\mu hbv_0\in\mathcal{B}_L$ with $h\ne0$, the vectors $e_\mu\pi^{\hat{\heh}[\nu]^+}(b)v_0$ span the vacuum space.
The theorem follows from (\ref{LWthm}) by comparing the dimensions of the subspace spanned by all vectors $h\cdot e_\mu\pi^{\hat{\heh}[\nu]^+}(b)v_0$ and the subspace spanned by all vectors $e_\mu hbv_0$ of fixed degree and weight.
\end{proof}
\end{comment}
The proof is parallel to that of Theorem 3.1 of \cite{BKP}.

\subsection{Parafermionic space and its current}
Recall that the map $\alpha\mapsto e_\alpha$ for $\alpha\in Q$ is extended to a projective representation of {\it Q} on $L(k\La_0)$.
This gives a diagonal action $\rho(k\alpha)=e_\alpha\ot\cdots\ot e_\alpha$ of the sublattice $kQ\subset Q$ such that $L(k\La_0)_\mu^{\hat{\heh}[\nu]^+}\rightarrow L(k\La_0)_{\mu+k\alpha}^{\hat{\heh}[\nu]^+}$.
We define the parafermionic space of the highest weight $k\La_0$ as the space of $kQ$-coinvariants in the $kQ$-module ${L(k\La_0)}^{\hat{\heh}[\nu]^+}$
\begin{align}
\label{parasp}
{L(k\La_0)}_{kQ}^{\hat{\heh}[\nu]^+}:=L(k\La_0)^{\hat{\heh}[\nu]^+}/{\rm span}_\C \{ (\rho(k\alpha)-1)\cdot v\mid \alpha \in Q,v \in L(k\La_0)^{\hat{\heh}[\nu]^+}\}.
\end{align}
We have the canonical projection
\begin{align}
\label{proj}
\pi_{kQ}^{\hat{\heh}[\nu]~+}:L(k\La_0)^{\hat{\heh}[\nu]^+} \rightarrow {L(k\La_0)}_{kQ}^{\hat{\heh}[\nu]^+}
\end{align}
and denote the composition $\pi_{kQ}^{\hat{\heh}[\nu]^+} \circ \pi^{\hat{\heh}[\nu]^+}:L(k\La_0) \rightarrow {L(k\La_0)}_{kQ}^{\hat{\heh}[\nu]^+}$ by $\pi$.
Note that in this case, we have 
\[{L(k\La_0)}_{kQ}^{\hat{\heh}[\nu]^+} \simeq \bigoplus_{\mu \in k\La_0 +Q/kQ}{L(k\La_0)}_{\mu}^{\hat{\heh}[\nu]^+}.\]

For every root $\beta$, we define the parafermionic current of charge {\it n} by
\begin{align}
\label{para}
\Psi_{n\beta}^{\hat{\nu}} (z) =\mathcal{Z}_{n\beta} (z)z^{-n\beta_{(0)} /k}
\ep_\beta^{-n/k},
\end{align}
where $\ep_\beta:L(k\La_0) \rightarrow \C^{\times}$ is given by
\[
\ep_\beta u=C(\beta,\mu)u\quad\text{for }u\in L(k\La_0)_\mu.
\]
\begin{comment}
In the same way as (\ref{1a}), we have \red{(?)}
\begin{align}
\label{nupsi}
\Psi_{\nu\alpha_i}^{\hat{\nu}}(z)=\mathcal{Z}_{\nu\alpha_i}(z)z^{-\nu (\alpha_i)_{(0)}/k}\ep_{\nu\alpha_i}^{1/k}=\mathcal{Z}_{\alpha_i}(\eta^{-1}z)z^{-(\alpha_i)_{(0)}}\ep_{\nu\alpha_i}^{1/k}=\Psi_{\alpha_i}^{\hat{\nu}}(\eta^{-1}z)\ep_{\nu\alpha_i-\alpha_i}^{1/k}.
\end{align}
\end{comment}
Since $\mathcal{Z}$-operators commute with the action of the Heisenberg subalgebra $\hat{\heh}[\nu]_{\frac1{r}\Z}$, the parafermionic current preserves the vacuum space $L(k\La_0)^{\hat{\heh}[\nu]^+}$.
The commutation relation (\ref{ex}) can be written as
\[x_\beta^{\hat{\nu}}(z)e_\alpha=C(\alpha,\beta)^{-1}e_\alpha x_\beta^{\hat{\nu}}(z)z^{\pair{\alpha,\beta_{(0)}}}.\]
From this relation and the one between $z^{\mu}$ and $e_\alpha$, we have
\[
[\rho(k\alpha),\Psi_{n\beta}^{\hat{\nu}}(z)]=0.
\]
Therefore, $\Psi^{\hat{\nu}}$ is well-defined on the parafermionic space ${L(k\La_0)}_{kQ}^{\hat{\heh}[\nu]^+}$.
For a quasi-particle of charge-type $\mathcal{R}^\prime=(n_{r_l^{(1)}},\cdots,n_{1,1})$, we define the parafermionic current of charge-type $\mathcal{R}^\prime$ by
\[
\Psi_{\mathcal{R}^\prime}^{\hat{\nu}}(z_{r_l^{(1)},l},\cdots,z_{1,1})=\mathcal{Z}_{\mathcal{R}^\prime}(z_{r_l^{(1)},l},\cdots,z_{1,1})z_{r_l^{(1)},l}^{-n_{r_l^{(1)},l}(\alpha_l)_{(0)}/k}\cdots z_{1,1}^{-n_{1,1}(\alpha_1)_{(0)}/k}\ep_{\alpha_l}^{-n_{r_l^{(1)},l}/k}\cdots\ep_{\alpha_1}^{-n_{1,1}/k}.
\]
Note that the parafermionic current of charge-type $\mathcal{R}^\prime$ also commutes with the diagonal action $\rho(k\alpha)$ for $\alpha\in Q$.
As in the $\mathcal{Z}$-operator, we set
\[\Psi_{\mathcal{R}^\prime}^{\hat{\nu}}(z_{r_l^{(1)},l},\cdots,z_{1,1})=\sum_{m_{r_l^{(1)},l},\cdots,m_{1,1}}\psi_{\mathcal{R}^\prime}^{\hat{\nu}}(m_{r_l^{(1)},l},\cdots,m_{1,1})z_{r_l^{(1)},l}^{-m_{r_l^{(1)},l}-n_{r_l^{(1)},l}}\cdots z_{1,1}^{-m_{1,1}-n_{1,1}}\]
where the summation is over all sequences $(m_{r_l^{(1)},l},\cdots,m_{1,1})$ such that $m_{p,i}\in\rho_i\Z+\frac{n_{p,i}\pair{(\alpha_i)_{(0)},\mu}}{k}$ on the $\mu$-weight space $L(k\La_0)_\mu^{\hat{\heh}[\nu]^+}$.

The following lemma associates the coefficients of $\mathcal{Z}$-operators with those 
of parafermionic currents.
\begin{lemma}
For a simple root $\beta$, $m\in\frac{1}{r}\Z$ and weight $\mu$ we have
\begin{align*}
\left.\mathcal{Z}_\beta(m)\mathrel{}\middle|\mathrel{}_{L(k\La_0)_\mu^{\hat{\heh}[\nu]^+}}=C(\beta,\mu)\psi_\beta^{\hat{\nu}}(m+\pair{\beta_{(0)},\mu}/k)\mathrel{}\middle|\mathrel{}_{L(k\La_0)_\mu^{\hat{\heh}[\nu]^+}}\right.
\end{align*}
\end{lemma}
\begin{proof}
By applying (\ref{para}) to the $\mu$-weight space $L(k\La_0)_\mu^{\hat{\heh}[\nu]^+}$, we have
\[
\left.\Psi_\beta^{\hat{\nu}}(z)\mathrel{}\middle|\mathrel{}_{L(k\La_0)_\mu^{\hat{\heh}[\nu]^+}}=C(\beta,\mu)^{-1/k}\mathcal{Z}_\beta(z)z^{-\pair{\beta_{(0)},\mu}/k}\mathrel{}\middle|\mathrel{}_{L(k\La_0)_\mu^{\hat{\heh}[\nu]^+}}\right..
\]
Hence we obtain the statement by taking the coefficient of $z^{-m-\pair{\beta_{(0)},\mu}/k-1}$.
\end{proof}
Next we consider the relations between different parafermionic currents.
We have the following lemma by direct computation. Noting that $C(\beta,\beta)=1$
for a simple root $\beta$, the proof is parallel to that of
Lemma 3.2 of \cite{BKP}.
\begin{lemma} \label{lem:Psi}
For a simple root $\beta$ and a positive integer $n$,
\begin{align}
\label{nb}
\left.\Psi_{n\beta}^{\hat{\nu}}(z)=\left(\prod_{1\le p < s\le n}\prod_{i=0}^{r-1} (z_s^{\frac1{r}}-\eta^iz_p^{\frac1{r}})^{\pair{\nu^i\beta,\beta}/k}\right)\Psi_\beta^{\hat{\nu}}(z_n)\cdots \Psi_\beta^{\hat{\nu}}(z_1)\mathrel{}\middle|\mathrel{}_{z_n=\cdots=z_1=z}\right..
\end{align}
\end{lemma}
We set
\[\Psi_{n_t\beta_t,\ldots,n_1\beta_1}^{\hat{\nu}}(z_t,\ldots,z_1)=\mathcal{Z}_{(n_t,\ldots,n_1)}(z_t,\ldots,z_1)\prod_{i=1}^tz_i^{-n_i(\beta_i)_{(0)}/k}\ep_{\beta_i}^{1/k}\]
for simplicity for a given simple roots $\beta_t,\ldots,\beta_1$ and charges $n_1,\ldots,n_1$.
Analogously to Lemma \ref{lem:Psi}, we obtain the following lemma. The proof is 
parallel to that of Lemma 3.3 of \cite{BKP}.
\begin{lemma}
For given simple roots $\beta_t,\ldots,\beta_1$ and positive integers $n_t,\ldots,n_1$,
\begin{align}
&\Psi_{n_t\beta_t,...,n_1\beta_1}^{\hat{\nu}}(z_t,...,z_1) \nonumber\\
&=\left(\prod_{1\le p<s\le t}C(\beta_s,\beta_p)^{n_sn_p/k}\prod_{i=0}^{r-1} (z_s^{\frac1{r}}-\eta^iz_p^{\frac1{r}})^{\pair{n_s\nu^i\beta_s,n_p\beta_p}/k}\right)\Psi_{n_t\beta_t}^{\hat{\nu}}(z_t)\cdots\Psi_{n_1\beta_1}^{\hat{\nu}}(z_1). \label{nb2}
\end{align}
\begin{comment}
Moreover, we have
\[\Psi_{n_t\beta_r,...,n_1\beta_1}^{\hat{\nu}}(z_t,...z_1)\]
\begin{eqnarray}
\label{nb3}
=\left(\prod_{p=1}^{t-1}C(\beta_t,\beta_p)^{-n_tn_p/k}\prod_{i=0}^{r-1}(z_t^{\frac1{r}}-\eta^iz_p^{\frac1{r}})^{\pair{n_t\nu^i\beta_t,n_p\beta_p}/k}\right)\Psi_{n_t\beta_t}^{\hat{\nu}}(z_t)\Psi_{n_{t-1}\beta_{t-1},...n_1\beta_1}^{\hat{\nu}}(z_{t-1},...,z_1).
\end{eqnarray}
\end{comment}
\end{lemma}
From Theorem \ref{th:main2}, we have
\begin{theorem} \label{th:main3}
For the highest weight $k\La_0$, the set of vectors
\[\pi_{kQ}^{\hat{\heh}[\nu]^+}\mathcal{Z}_{\mathcal{R}^\prime}(m_{r_l^{(1)},l},\cdots,m_{1,1})v_0=\psi_{\mathcal{R}^\prime}^{\hat{\nu}}(m_{r_l^{(1)},l},\cdots,m_{1,1})v_0\]
is a basis of the parafermionic space $L(k\La_0)_{kQ}^{\hat{\heh}[\nu]^+}$, where $\mathcal{Z}_{\mathcal{R}^\prime}(m_{r_l^{(1)},l},\cdots,m_{1,1})v_0$ is a vector that appears in a basis of the vacuum space $L(k\La_0)^{\hat{\heh}[\nu]^+}$.
\end{theorem}

\section{parafermionic character formula}
\subsection{Grading operator}
Since we do not find the coset Virasoro algebra construction \cite[\S3]{Li} for
the twisted vertex operator case, we introduce, by hand, a grading operator for our
parafermionic space $L(k\La_0)_{kQ}^{\hat{\heh}[\nu]^+}$.
Define an operator $D$ acting on the space $L(k\La_0)^{\hat{\heh}[\nu]^+}$
as follows.
\[
D=-d-D^{\hat{\heh}[\nu]^+},\qquad
\left.D^{\hat{\heh}[\nu]^+}\right|_{L(k\La_0)^{\hat{\heh}[\nu]^+}_\mu}
=\frac{\pair{\mu_{(0)},\mu_{(0)}}}{2k}.
\]
%
\begin{comment}
We need to modify the grading operator on $L(k\La_0)$ given by the Virasoro action.
In this case, the grading operator $L_\Omega(0)={\rm Res}_{z=0}Y_\Omega(\omega_\Omega,z)$ on $L(k\La_0)$ called the parafermionic grading operator is defined by
\[Y_\Omega(\omega_\Omega,z)=\sum_{n\in\Z}L_\Omega(n)z^{-n-2},\quad\omega_\Omega=\omega-\omega_{\bf h}\]
where $\omega$ is the conformal vector that induces the Virasoro action and
\[\omega_{\bf h}=\frac{1}{2k}\sum_{i=1}^{{\rm dim}\heh}h_i(-1)h_i(-1)\]
for an orthonormal basis of $\heh$ with respect to the bilinear form $\pair{\cdot,\cdot}$.
From \cite[Section 3 and Theorem 6.4]{Li} , we have
\end{comment}
%
Then, for a simple root $\beta\in L$ and $m\in\Z$, we have
\begin{align}\label{paragr}
[D,x_\beta^{\hat{\nu}}(m)]=\left(-m-\frac{\pair{\beta_{(0)},\beta_{(0)}}}{2k}\right)x_\beta^{\hat{\nu}}(m).
\end{align}
On $v\in L(k\La_0)^{\hat{\heh}[\nu]^+}$, we also have
\[
[D,\psi_\beta^{\hat{\nu}}(m)]=\left(-m-\frac{\pair{\beta_{(0)},\beta_{(0)}}}{2k}\right)\psi_\beta^{\hat{\nu}}(m).
\]
We call the coefficient of the right hand side the conformal energy of $\psi_\beta^{\hat{\nu}}(m)$ and write
\begin{align}\label{conen}
{\rm en}\,\psi_\beta^{\hat{\nu}}(m)=-m-\frac{\pair{\beta_{(0)},\beta_{(0)}}}{2k}.
\end{align}

Now we compute the conformal energies of $\psi_{n\beta}^{\hat{\nu}}(m)$ and $\psi_{n_t\beta_t,\ldots,n_1\beta_1}^{\hat{\nu}}(m_t,\ldots,m_1)$.
\begin{lemma}
For a simple root $\beta$ and charge {\it n}, we have
\begin{align}\label{conen2}
{\rm en}\,\psi_{n\beta}^{\hat{\nu}}(m)=-m-\frac{n^2\pair{\beta_{(0)},\beta_{(0)}}}{2k}.
\end{align}
Moreover, for simple roots $\beta_t,\cdots,\beta_1$ and charge $n_t,\ldots,n_1$, we have
\begin{align}\label{conen3}
{\rm en}\,\psi_{n_t\beta_t,\ldots,n_1\beta_1}^{\hat{\nu}}(m_t,\ldots,m_1)=\sum_{i=1}^t\left({\rm en}\,\psi_{n_i\beta_i}^{\hat{\nu}}(m_i)-\sum_{p=1}^{i-1}\frac{\pair{n_i(\beta_i)_{(0)},n_p(\beta_p)_{(0)}}}{k}\right).
\end{align}
\end{lemma}
\begin{proof}
Consider the right hand side of (\ref{nb}).
This means that $\psi_{n\beta}^{\hat{\nu}}(m)$ can be expressed by $\psi_\beta^{\hat{\nu}}(m_i)$ with $i=1,\cdots,n$.
Note that each term $z_i^{\pair{\beta_{(0)},\beta_{(0)}}/k}$ reduces the conformal energy by $\pair{\beta_{(0)},\beta_{(0)}}/k$.
Since (\ref{nb}) contains $n(n-1)/2$ such terms and the energy of $\psi_\beta^{\hat{\nu}}(m)$ is give by (\ref{conen}), we have
\begin{align*}
{\rm en}\,\psi_{n\beta}^{\hat{\nu}}(m)&=-\sum_{i=1}^n\left(m_i+\frac{\pair{\beta_{(0)},\beta_{(0)}}}{2k}\right)-\frac{n(n-1)}{2}\cdot\frac{\pair{\beta_{(0)},\beta_{(0)}}}{k}\\
&=-m-\frac{n^2\pair{\beta_{(0)},\beta_{(0)}}}{2k}
\end{align*}
where $m_1+\cdots+m_n=m$.
Now, (\ref{conen3}) follows from (\ref{nb2}) and the same argument by using the energy of $\psi_{n\beta}^{\hat{\nu}}(m)$ given by (\ref{conen2}).
\end{proof}

Using \eqref{ed} we can show $[D,\rho(k\alpha)]=0$ for ($\alpha\in Q$).
Hence, the grading operator $D$ is well defined on the parafermionic space 
$L(k\La_0)_{kQ}^{\hat{\heh}[\nu]^+}$.

\subsection{Character formula}
We define the character of the parafermionic space ${L(k\La_0)}_{kQ}^{\hat{\heh}[\nu]^+}$ by
\[
{\rm ch}\,{L(k\La_0)}_{kQ}^{\hat{\heh}[\nu]^+}=\sum_{m,r_1,\ldots r_l\geq0}{\rm dim}({L(k\La_0)}_{kQ}^{\hat{\heh}[\nu]^+})_{(m,r_1,\ldots,r_l)}q^my_1^{r_1}\cdots y_l^{r_l},
\]
where $({L(k\La_0)}_{kQ}^{\hat{\heh}[\nu]^+})_{(m,r_1\ldots,r_l)}$ is the weight space spanned by monomial vectors of conformal energy $-m$ and color-type $(r_1,\ldots,r_l)$ (see \eqref{col type}).
Consider an arbitrary quasi-particle monomial
\begin{align}
\label{vector}
x_{n_{r_l^{(1)},l}\alpha_l}^{\hat{\nu}}(m_{r_l^{(1)},l})\cdots x_{n_{1,l}\alpha_l}^{\hat{\nu}}(m_{1,l})\cdots x_{n_{r_1^{(1)},1}\alpha_1}^{\hat{\nu}}(m_{r_1^{(1)},1})\cdots x_{n_{1,1}\alpha_1}^{\hat{\nu}}(m_{1,1})\in B_W^\prime.
\end{align}
Denote by
\[\mathcal{R}^\prime=(n_{r_l^{(1)},l},\cdots,n_{1,1}),\quad\mathcal{R}=(r_l^{(1)},\cdots,r_1^{(k-1)})\]
its charge-type, dual-charge-type.
We define $\mathcal{P}_i=(p_i^{(1)},\ldots,p_i^{(k-1)})$ by $p_i^{(s)}=r_i^{(s)}-r_i^{(s+1)}$ ($r_i^{(k)}=0$) 
for $i=1,\ldots,l$, $s=1,\ldots,k-1$, so that $p_i^{(s)}$ stands for the number of quasi-particles of color $i$ and charge $s$ in the monomial (\ref{vector}).
Set ${\mathcal P} = ({\mathcal P}_l,\ldots,{\mathcal P}_1)$.
To emphasize the dependence of $k$, we also write ${\mathcal P}^{(k-1)}$.
We consider the parafermionic space basis given in Theorem \ref{th:main3}.
From \eqref{conen3}, the conformal energy of 
\[
\psi_{n_{r_l^{(1)},l}\alpha_l,\ldots,n_{1,1}\alpha_1}^{\hat{\nu}}(m_{r_l^{(1)},l},\ldots,m_{1,1})
\]
is equal to
\begin{align}
&-\sum_{i=1}^l\left(\sum_{s=1}^{r_i^{(1)}}m_{s,i}+\sum_{s=1}^{r_i^{(1)}}\frac{n_{s,i}^2\rho_i}{k}+\sum_{s=1}^{r_i^{(1)}}\left(\sum_{t=1}^{s-1}\frac{2n_{t,i}n_{s,i}\rho_i}{k}+\sum_{j=1}^{i-1}\sum_{t=1}^{r_j^{(1)}}\frac{\pair{n_{s,i}(\alpha_i)_{(0)},n_{t,j}(\alpha_j)_{(0)}}}{k}\right)\right) \nonumber \\
&=-\sum_{i=1}^l\sum_{s=1}^{r_i^{(1)}}m_{s,i}-\frac12
\sum_{i,j=1}^l\sum_{s=1}^{r_i^{(1)}}\sum_{t=1}^{r_j^{(1)}}
\frac{\pair{n_{s,i}(\alpha_i)_{(0)},n_{t,j}(\alpha_j)_{(0)}}}{k}, \label{vecen}
\end{align}
where $\rho_i$ is defined in \eqref{rho}. Since
\[
\sum_{s=1}^{r_i^{(1)}}n_{s,i}=\sum_{s=1}^{k-1}sp_i^{(s)},
\]
\eqref{vecen} is further calculated as
\begin{equation} \label{vecen2}
-\sum_{i=1}^l\sum_{s=1}^{r_i^{(1)}}m_{s,i}-\frac12
\sum_{i,j=1}^l\sum_{s,t=1}^{k-1}
\frac{st}{k}\pair{(\alpha_i)_{(0)},(\alpha_j)_{(0)}}p_i^{(s)}p_j^{(t)}.
\end{equation}
%
\begin{comment}
From definition it is obvious that the integers $\sum_{j=r_i^{(m+1)}+1}^{r_i^{(m)}}n_{j,i}$ can be expressed as $mp_i^{(m)}$. 
Therefore we have that conformal energy in (\ref{vecen}) is written by
\begin{align}
\label{revecen}
=-\sum_{i=1}^l\sum_{j=1}^{r_i^{(1)}}m_{j,i}-\sum_{i,j=1}^l\sum_{s,t=1}^{k-1}\frac{\pair{(\alpha_i)_{(0)},(\alpha_j)_{(0)}}}{2k}stp_i^{(s)}p_j^{(t)}.
\end{align}
Recall that the conditions on energies of monomials in $B_W$ is written by the dual-charge-type
\[\pair{(\alpha_i)_{(0)},(\alpha_{i-1})_{(0)}}\sum_{p=1}^{r_i^{(1)}}\sum_{q=1}^{r_{i-1}^{(1)}}{\rm min}\{n_{q,i-1},n_{p,i}\}=\pair{(\alpha_i)_{(0)},(\alpha_{i-1})_{(0)}}\sum_{s=1}^kr_{i-1}^{(s)}r_i^{(s)},\]
\[\rho_i\sum_{p=1}^{r_i^{(1)}}\left(\sum_{p>p^\prime>0}2{\rm min}\{n_{p,i},n_{p^\prime,i}\}+n_{p,i}\right)=\rho_i\sum_{s=1}^k{r_i^{(s)}}^2.\]
See \cite{BS}.
Furthermore we have
\begin{align}\label{rp}
\nonumber\sum_{s=1}^kr_i^{(s)}r_j^{(s)}&=\sum_{s=1}^k(p_i^{(s)}+\cdots+p_i^{(k)})(p_j^{(s)}+\cdots+p_j^{(k)})\\
&=\sum_{s,t=1}^k{\rm min}\{s,t\}p_i^{(s)}p_j^{(t)}.
\end{align}
\end{comment}
%

To calculate the character of the parafermionic space, we use the corresponding result
of the principal subspace.
\begin{theorem} \label{th:BS2}
$($\cite[Theorem 6.1]{BS}$)$ 
For each of the affine Lie algebras $A_{2l-1}^{(2)}$, $D_{l+1}^{(2)}$, $E_6^{(2)}$, $D_4^{(3)}$, we have
\[{\rm ch}\,{W(k\La_0)}=\sum_{\mathcal{P}^{(k)}}\frac{q^{\frac1{2}\sum_{i,j=1}^l\pair{(\alpha_i)_{(0)},(\alpha_j)_{(0)}}\sum_{s,t=1}^k\min\{s,t\}p_i^{(s)}p_j^{(t)}}}{\prod_{i=1}^l\prod_{s=1}^{k} (q^{\rho_i})_{p_i^{(s)}}}\prod_{i=1}^l y_i^{\sum_{s=1}^ksp_i^{(s)}}\]
where the sum runs over all sequences ${\mathcal P}^{(k)}$ of $lk$ nonnegative integers.
\end{theorem}
Note that in this formula, we use the index set $\mathcal{P}^{(k)}$ rather than 
$\mathcal{P}^{(k-1)}$, since the basis vectors contain quasi-particles of charge $k$
by Theorem \ref{th:BS}. Now we can obtain the character of the parafermionic space $L(k\La_0)_{kQ}^{\hat{\heh}[\nu]^+}$.
\begin{theorem}
For each of the affine Lie algebras $A_{2l-1}^{(2)}$, $D_{l+1}^{(2)}$, $E_6^{(2)}$, $D_4^{(3)}$, we have
\[{\rm ch}\,{L(k\La_0)}_{kQ}^{\hat{\heh}[\nu]^+}=\sum_{\mathcal{P}^{(k-1)}}\frac{q^{\frac1{2}\sum_{i,j=1}^l\pair{(\alpha_i)_{(0)},(\alpha_j)_{(0)}}\sum_{s,t=1}^{k-1}D_{s,t}^{(k)}p_i^{(s)}p_j^{(t)}}}{\prod_{i=1}^l\prod_{s=1}^{k-1} (q^{\rho_i})_{p_i^{(s)}}}\prod_{i=1}^l y_i^{\sum_{s=1}^{k-1}sp_i^{(s)}}\]
where the sum runs over all sequences ${\mathcal P}^{(k-1)}$ of $l(k-1)$ nonnegative integers and
\[
D_{s,t}^{(k)}={\rm min}\{s,t\}-\frac{st}{k}.
\]
\end{theorem}
\begin{proof}
Comparing Theorems \ref{th:BS} and \ref{th:BS2}, one can readily calculate the 
generating function with weight given by the first term of \eqref{vecen2} as
\[
\sum_{\mathcal{P}^{(k-1)}}\frac{q^{\frac1{2}\sum_{i,j=1}^l\pair{(\alpha_i)_{(0)},(\alpha_j)_{(0)}}\sum_{s,t=1}^{k-1}{\rm min}\{s,t\}p_i^{(s)}p_j^{(t)}}}{\prod_{i=1}^l\prod_{s=1}^{k-1}(q^{\rho_i}
)_{p_i^{(s)}}}\prod_{i=1}^l y_i^{\sum_{s=1}^{k-1}sp_i^{(s)}}.
\]
Taking the conformal shift, namely the second term of \eqref{vecen2}, into account,
we obtain the desired formula.
\end{proof}
\begin{remark}
We compare this result with \cite[Conjecture 5.3]{HKOTT}. First, note that due to the
difference of the normalization of the bilinear form $\pair{\,,\,}$ on the root lattice $Q$,
$q$ in this paper should be replaced with $q^r$ to compare with \cite{HKOTT}. 
Recalling $r\rho_i=t_i^{\lor}$ and removing the factor corresponding to the Heisenberg
subalgebra $\hat{\heh}[\nu]_{\frac1{r}\Z}$, we see that the above formula is consistent
with Conjecture 5.3 in \cite{HKOTT}.
\end{remark}

\section*{Acknowledgments}
M.O. is supported by Grants-in-Aid for Scientific Research No.~19K03426
and No.~16H03922 from JSPS.
This work was partly supported by Osaka City University Advanced Mathematical Institute (MEXT Joint Usage/Research Center on Mathematics and Theoretical Physics JPMXP0619217849).


\begin{thebibliography}{99}
\bibitem{B}M. Butorac, A note on principal subspaces of the affine Lie algebras in types $B_l^{(1)}$, $C_l^{(1)}$, $F_4^{(1)}$ and $G_2^{(1)}$, Comm. Algebra {\bf48}, 5343-5359 (2020).
\bibitem{BK}M. Butorac, S. Ko\v{z}i\'{c}, Principal subspaces for the affine Lie algebras in types $D$, $E$ and $F$, preprint arXiv:1902.10794.
\bibitem{BKP}M. Butorac, S. Ko\v{z}i\'{c} and M. Primc, Parafermionic bases of standard modules for affine Lie algebras, Mathematische Zeitschrift (2020), published online,
https://doi.org/10.1007/s00209-020-02639-w.
\bibitem{BS}M. Butorac and C. Sadowski, Combinatorial bases of principal subspace of modules for twisted affine Lie algebras of type $A_{2l-1}^{(2)}$, $D_l^{(2)}$, $E_6^{(2)}$ and $D_4^{(3)}$, New York J. Math. {\bf 25} (2019), 71-106.
\bibitem{CalLM}C. Calinescu, J. Lepowsky and A. Milas, Vertex-algebraic structure of principal subspaces of standard $A_2^{(2)}$-modules, I, Internat.~J.~Math.~{\bf 25} (2014),
1450063.
\bibitem{DL}C. Dong and J. Lepowsky, The algebraic structure of relative twisted vertex operators, J.~Pure~Appl.~Algebra {\bf 110} (1996), 259-295.
%\bibitem{DL2}C. Dong and J. Lepowsky, Generalized Vertex Algebras and Relative Vertex Operators. Progress in Math. {\bf112} Birkhauser 1993.
\bibitem{FHK}I. Frenkel, Y.-Z. Huang, J. Lepowsky, On Axiomatic Approaches to Vertex Operator Algebras and Modules, Mem. Amer. Math. Soc. {\bf104}, No. 494 (1993), 64 pages.
\bibitem{FK}I.B. Frenkel, V.G.  Kac, Basic representations of affine Lie algebras and dual resonance models, Invent. Math. {\bf 62}, 23-66 (1980).
\bibitem{FS1}B. Feigin and A. Stoyanovsky, Quasi-particles models for the representations of Lie algebras and geometry of flag manifold, arXiv:hep-th/9308079.
\bibitem{FS2}B. Feigin and A. Stoyanovsky, Functional models for representations of current algebras and semi-infinite Schubert cells (Russian), Funktsional~Anal~i~Prilozhen, {\bf 28} (1994), 68-90; translation in: Funct.~Anal.~Appl., {\bf 28} (1994), 55-72.
\bibitem{FLM}I. Frenkel, J. Lepowski and A. Meruman, Vertex operator calculas, in: Mathematical Aspects of String Theory, Proc. 1986 Conference, San Diego, ed. by S.-T. Yau, Would Scientific, Singapore, 1987, 150-188.
\bibitem{G1}G. Georgiev, Combinatorial constructions of modules for  infinite-dimensional Lie algebras, I, Principal subspace, J. Pure Appl. Algebra {\bf 112} (1996), 247-286.
\bibitem{G2}G. Georgiev, Combinatorial constructions of modules for  infinite-dimensional Lie algebras, II, Parafermionic space, arXiv:q-alg/9504024.
\bibitem{HKOTT}G. Hatayama, A. Kuniba, M. Okado, T. Takagi and T. Tsuboi, Path, Crystals and Fermionic Formulae, MathPhys Odyssey 2001, 205--272, 
Prog.\ Math.\ Phys.\ {\bf 23}, Birkh\"auser Boston, Boston, MA, 2002.
\bibitem{Kac}V.G. Kac, Infinite dimensional Lie algebras, 3rd ed., Cambridge University Press, Cambridge, 1990.
\bibitem{KNS}A. Kuniba, T. Nakanishi and J. Suzuki, Characters in conformal field 
theories from thermodynamic Bethe Ansatz, Modern Phys. Lett. {\bf A8} (1993),
1649-1659.
\bibitem{L}J. Lepowsky, Calculus of twisted vertex operators, Proc.~Nat.~Acid.~Sci.~USA, {\bf 82} (1985), 8295-8299.
\bibitem{LP}J. Lepowsky and M. Primc, Structure of the standard modules for the affine Lie algebra $A^{(1)}_1$, Contemporary Math., {\bf 46}, Amer. Math. Soc, Providence, RI, 1985.
\bibitem{LW}J. Lepowsky and R.L. Wilson, The structure of standard modules, I: Universal algebras and the Rogers-Ramanujan identities, Invent. Math. {\bf 77} (1984), 199-290.
\bibitem{Li}H.-S. Li, On abelian coset generalized vertex algebras, Commun. Contemp. Math. {\bf 03}, No. 02, (2001), 287-340.
\bibitem{PS1}M. Penn and C. Sadowski, Vertex-algebraic structure of principal subspace of basic $D_4^{(3)}$-modules, Ramanujan~J. {\bf 43} (2017), 571-617.
\bibitem{PS2}M. Penn and C. Sadowski, Vertex-algebraic structure of principal subspace of basic modules for twisted affine Kac-Moody Lie algebras of type $A_{2n+1}^{(2)}$, $D_n^{(2)}$, $E_6^{(2)}$, Journal of Algebra, {\bf 496} (2018), 242-291.
\bibitem{P} M. Primc, Vertex operator construction of standard modules for $A^{(1)}_n$, Pacific J. Math., {\bf 162} (1994), 143-187.
\end{thebibliography}
\end{document}